\documentclass[12pt]{amsart}

\usepackage{amssymb, amsmath, stmaryrd, mathrsfs, hyperref}
\usepackage{tikz} \usetikzlibrary{matrix,arrows,decorations.pathmorphing}

\usepackage{enumitem}

\makeatletter
\@namedef{subjclassname@2020}{%
  \textup{2020} Mathematics Subject Classification}
\makeatother

\usepackage[utf8]{inputenc}
\usepackage[T1]{fontenc}

\newtheorem{theorem}{Theorem}[section]
\newtheorem{corollary}[theorem]{Corollary}
\newtheorem{lemma}[theorem]{Lemma}

\newtheorem{notation}[theorem]{Notation}
\newtheorem{fact}[theorem]{Fact}
\newtheorem{terminology}[theorem]{Terminology}
\newtheorem{assumption}[theorem]{Assumption}

\theoremstyle{definition}
\newtheorem{definition}[theorem]{Definition}
\newtheorem{remark}[theorem]{Remark}
\newtheorem{example}[theorem]{Example}

\numberwithin{equation}{section}

\newcommand{\norm}[1]{\left\lVert#1\right\rVert}

\newcommand{\bc}{\mathbf{c}}
\newcommand{\be}{\mathbf{e}}
\newcommand{\bff}{\mathbf{f}}
\newcommand{\bY}{\mathbf{Y}}

\newcommand{\bX}{\mathbf{X}}
\newcommand{\bU}{\mathbf{U}}
\newcommand{\bV}{\mathbf{V}}
\newcommand{\bW}{\mathbf{W}}
\newcommand{\bS}{\mathbf{S}}
\newcommand{\bE}{\mathbf{E}}
\newcommand{\bG}{\mathbf{G}}
\newcommand{\bGamma}{\mathbf{\Gamma}}

\newcommand{\bpi}{\boldsymbol{\pi}}
\newcommand{\bmu}{\boldsymbol{\mu}}

\newcommand{\Z}{\mathbb{Z}}
\newcommand{\R}{\mathbb{R}}
\newcommand{\Q}{\mathbb{Q}}

\newcommand{\A}{\mathbb{A}}

\newcommand{\N}{\mathbb{N}}

\newcommand{\W}{\mathscr{W}}

\newcommand{\cC}{\mathcal{C}}

\newcommand{\cL}{\mathcal{L}}

\newcommand{\trdeg}{\mathop{\rm tr.deg}\nolimits}
\newcommand{\dd}{\mathbf{d}}

\newcommand{\chr}{\text{\rm char}}

\newcommand{\Set}{{\rm Set}}

\newcommand{\tr}{{\rm tr}}
\newcommand{\F}{\mathbb{F}}

\newcommand{\Alg}{\text{\rm-Alg}}
\newcommand{\Rng}{\text{\rm-Rng}}
\newcommand{\Sch}{\text{\rm-Sch}}

\newcommand{\Hom}{\text{\rm Hom}}

\newcommand{\dcl}{dcl}
\newcommand{\acl}{acl}
\newcommand{\bdd}{bdd}

\newcommand{\Cb}{Cb}

\def\Ind#1#2{#1\setbox0=\hbox{$#1x$}\kern\wd0\hbox to 0pt{\hss$#1\mid$\hss}
\lower.9\ht0\hbox to 0pt{\hss$#1\smile$\hss}\kern\wd0}
\def\Notind#1#2{#1\setbox0=\hbox{$#1x$}\kern\wd0\hbox to 0pt{\mathchardef
\nn="3236\hss$#1\nn$\kern1.4\wd0\hss}\hbox to 0pt{\hss$#1\mid$\hss}\lower
.9\ht0\hbox to 0pt{\hss$#1\smile$\hss}\kern\wd0}
\def\ind{\mathop{\mathpalette\Ind{}}}
\def\nind{\mathop{\mathpalette\Notind{}}}

\begin{document}


\baselineskip=17pt


\title[Graphons arising from graphs definable over finite fields]{Graphons arising from graphs definable over finite fields}

\author[M. D{\v z}amonja]{Mirna D{\v z}amonja}
\address{IRIF (CNRS-Université de Paris)\\
        Bâtiment Sophie Germain, Case courrier 7014\\
        8 Place Aurélie Nemours,
        75205 Paris Cedex 13, France}
\email{mdzamonja@irif.fr}

\author[I. Toma{\v s}i{\'c}]{Ivan Toma{\v s}i{\'c}}
\address{School of Mathematical Sciences\\
  	Queen Mary University of London\\
        London, E1 4NS\\
        United Kingdom}
\email{i.tomasic@qmul.ac.uk}

\date{}

\begin{abstract}
We prove a version of Tao's algebraic regularity lemma for asymptotic classes in the context of graphons. We apply it to study expander difference polynomials over fields with powers of Frobenius.
\end{abstract}



\maketitle

\section{Introduction}

\subsection{Historical overview and the statement of results}
Tao's algebraic regularity lemma is a variant of the celebrated Szemer\'{e}di's regularity lemma that applies to graphs that can be defined by a first-order formula over finite fields. It states that such a graph can be decomposed into definable pieces which are roughly about the same size and such that the edges between those pieces behave almost randomly. This process is referred to as \emph{regularisation}.
It was proved by Tao in \cite{tao-original} in order to study expander polynomials over finite fields, and initially it was formulated for fields of large enough characteristics. 

Further developments on Tao's lemma have a somewhat complex history.
In a private correspondence to Tao, Hrushovski \cite{udi-letter} gave another proof of the lemma using the model-theoretic tools for studying the growth rates of definable sets over finite fields, as developed by Chatzidakis-van den Dries-Macintyre in \cite{CDM}. Independently, Pillay and Starchenko gave an analogous proof in the preprint \cite{pillay-starchenko}. The advantage of these proofs is that they remove the requirement of the large characteristics of the field.
Pillay and Starchenko state that their proof also works for `measurable'  structures studied in the the context of asymptotic classes of finite structures by Macpherson-Steinhorn \cite{dugald-charlie} and Elwes and Macpherson \cite{dugald-richard}. Garcia-Macpherson-Steinhorn \cite{GMS} state, without proof, a version of the lemma for classes of finite structures controlled by suitable pseudofinite dimensions, which potentially generalise asymptotic classes.

In this paper we give a proof of Tao's algebraic regularity lemma for finite fields of any characteristics using concepts from the theory of graphons, and we adapt the lemma to \emph{relative} asymptotic classes. We apply it 
to study expander difference polynomials over fields with powers of Frobenius.
As far as the algebraic regularity lemma for fields goes, our proof benefited from the strategy used in the revisited and simplified proof of his lemma by Tao in his blog article \cite{tao-blog}. 

Upon writing a preliminary version of this article, we received a copy of \cite{udi-letter}. We now understand that Hrushovski envisaged in that letter that Tao's regularity lemma applies to fields with Frobenius (\ref{frob-fields}, \ref{reg-frob} and that Tao's \emph{algebraic constraint} is directly related to the existence of a model-theoretic \emph{group configuration}.

\subsection{Graphons}\label{graphons-context}
The space of graphons with the cut metric is known as one of the most suitable contexts for studying the limit behaviour of (dense) finite graphs, as detailed in \cite{lovasz}. 

It is also extremely useful for formulating and proving regularity lemmas. Indeed, to express that a graph can be regularised to a certain precision, it is enough to say that it is sufficiently close to a \emph{stepfunction} in the space of graphons, a piecewise constant probability function on the unit square encoding a rectangular array of the edge densities of pieces forming a regular decomposition of the graph, see Definition~\ref{def-stepfun}.   

The algebraic regularity lemma in this context states the following, see theorems \ref{reg-asymp}, \ref{reg-cdm} and Corollary~\ref{accumul}.

\begin{quote}
\emph{In the space of graphons, the set of accumulation points of the family of realisations of a definable bipartite graph over the structures ranging in an asymptotic class is a finite set of stepfunctions.}
\end{quote}

\emph{Asymptotic classes} are defined as classes of finite structures such that growth rates of definable sets are controlled by rational monomials \ref{asymp-cl-def}. Consequently, stepfunctions appearing in the above take arbitrary rational probability values.  In contrast, several recent regularity lemmas from other contexts, including \cite{malliaris-shelah} in \emph{stable} and \cite{chernikov-starchenko} in \emph{distal} context can be viewed as 0-1 laws, and the resulting stepfunctions only take values 0 and 1.

\subsection{Definable regularisation}\label{uniform} An important feature of this note is a uniform formulation of the lemma. We emphasise that realisations of a single \emph{definable stepfunction} over finite structures from an asymptotic class can be used to regularise a definable graph. The proof is almost constructive, and we carefully analyse the space of parameters needed to define the regularising stepfunction. It becomes apparent that there are \emph{finitely many} possible asymptotic behaviours of its realisations as we vary the finite structures in an asymptotic class.

The strategy of proof is congenial to the context of graphons in view of the fact that we start with a general weak regularity Lemma~\ref{wk6-reg} for graphons, and then use definability and Chatzidakis-van den Dries-Macintyre properties to dramatically improve the conclusion.

\subsection{Fields with powers of Frobenius}\label{frob-fields}
We draw attention to the fact that the algebraic regularity lemma applies to finite-dimensional bipartite graphs definable in the language of difference rings over the class of algebraic closures of finite fields equipped with powers of the Frobenius automorphism, which was shown to be an asymptotic class in \cite{ive-mark} (with a suitable modification to the definition of `asymptotic class'). This is a significant generalisation of the context of finite fields and a potential new source of interesting examples. We outline the possibility of studying \emph{difference expander polynomials} and we describe the difference morphisms which are moderate expanders.

\subsection{Novelty in the paper}

The `bounded-complexity' statements  of Tao's regularity lemma in various contexts have been (independently) known to experts. Our main proof loosely follows the strategy of the proof from \cite{tao-blog}. 

The added value of this paper consists of the following contributions.
\begin{enumerate}
\item The uniform statement of the result as discussed in \ref{uniform}, for a given set of formulae defining the initial graph and without the use of the notion of bounded complexity of formulae, which makes the proof more constructive.
\item It is aesthetically convenient to state the regularity lemma in the space of graphons, as noted in \ref{graphons-context}.
\item The profound reason for working in the space of graphons is to fix the arguments from Tao's proof  \cite{tao-blog} in a single space within  the realm of  classical functional analysis. While each step in his proof is formulated for operators on finite-dimensional normed spaces, the dimensions grow with each step so it is difficult to track the key objects through the changes, or to give a bound on the dimensions of relevant spaces a priori in terms of the complexity of the formulae defining the given bipartite graph.
\item While finite fields are extensively used in combinatorics, fields with powers of Frobenius touched upon in \ref{frob-fields} are less known. We hope that our treatment will introduce them to a wider audience. 
\item As pointed out by Hrushovski in \cite{udi-letter}, as well as Tao in a number of talks, the hidden group structure in `solving the algebraic constraint'-type theorems needed to discuss expanding properties of polynomials can be uncovered using the model-theoretic \emph{group configuration} arguments. We present these ideas as explicitly as possible in 
\ref{alg-constraint} whilst proving the expanding dichotomy for `difference polynomials'.

\end{enumerate}

\subsection{Future work}

Tao's lemma, as much as the original Szemer\'{e}di's lemma, is only interesting in the context of reasonably dense graphs. 
After much research on obtaining an analogue of Szemer\'{e}di's lemma for sparse graphs, as reported by Kohayakawa in \cite{kohayakawa}, such an analogue was found by Scott in \cite{scott}. In further work we hope to extend our ideas from graphons to the Ne\v{s}et\v{r}il-Ossona de Mendez notion of \emph{measurings} \cite{unified} in order to obtain local versions of Tao's regularity lemma which would deal with sparse definable graphs. This seems like a viable programme in the view of  Scott's lemma \cite{scott} and the fact that measurings give a unified approach to combinatorial limits which has graphons as a special case, but also applies to classes of sparse graphs.

\subsection{Acknowledgement}
We would like to thank the anonymous referee for their dedicated work on the paper and numerous suggestions for improvements.

\section{Background: Graphs and Graphons}

The background on the space of graphons given in this section is based almost entirely on the book \cite{lovasz} by Lov\'asz, presented with minor modifications needed to treat our `bipartite/non-symmetric'  version of kernels and graphons. 

\subsection{Kernels and graphons}
\begin{definition}
\begin{enumerate}
\item The space of \emph{kernels} $\tilde{\W}$ is the space of the equivalence classes of essentially bounded measurable functions $[0,1]^2\to\R$ modulo equivalence up to a measure 0 set, i.e., the underlying vector space of $\tilde{\W}$ is $L^\infty([0,1]^2)$. Following the usual conventions, we often abuse the notation and speak of `functions' or `kernels' (and `graphons' below) in terms of functions rather than the equivalence classes of functions, by taking representatives and by using the notation $\W$ in place of $\tilde{\W}$. Sometimes we need to be more explicit, for example
a suitable metric on $\tilde{\W}$ will be introduced below, but it is only a pseudo-metric on $\W$.
\item The space of \emph{graphons} is
$$
\W_0=\{W\in\W:0\leq W\leq 1\}.
$$
\item We write
$$
\W_1=\{W\in\W:-1\leq W\leq 1\}.
$$
\end{enumerate}
The inequalities are given up to a measure zero set (see Definition \ref{measurezero}).
\end{definition}

\begin{remark}\label{comp-symm} The classical literature on graphons \cite{lovasz} usually stipulates that kernels and graphons have to be symmetric functions, and we write $\W^{\text{sym}}, \W_0^{\text{sym}}$ and  $\W_1^{\text{sym}}$ for the corresponding spaces. 

Given that we are primarily interested in studying the limits of bipartite and directed graphs, we \emph{do not} assume the symmetry, but our context is easily reconciled with the classical one through the map 
$$
\W\to\W^\text{sym},
$$
assigning 
to every $W\in \W$ the symmetric `bipartite kernel' 
$$
W^{\text{sym}}(x,y)=\begin{cases} W(2x,2y-1) & \text{ when } (x,y)\in[0,1/2]\times[1/2,1],\\
W^*(2x-1,2y) & \text{ when } (x,y)\in[1/2,1]\times[0,1/2],\\
0 & \text{ elsewhere}.
\end{cases}
$$
Here $W^\ast$ denotes \emph{transpose} kernel of $W$, defined by
$$
W^*(x,y)=W(y,x).
$$ 
The above map will yield a closed embedding in the cut metric from \ref{cut-norm-metric} and will preserve the membership in $\W, \W_0$ and  $\W_1$, which allows us to resort to the classical theory of symmetric graphons whenever needed.
\end{remark}

\begin{definition}\label{measurezero}
Let $U$, $W$ be two kernels. 
\begin{enumerate}
\item We say that $U$ and $W$ are \emph{equal almost everywhere} if the set $\{(x,y)\in[0,1]^2:U(x,y)\neq W(x,y)\}$ is a null set with respect to the Lebesgue measure.
\item We say that $U$ and $W$ are \emph{isomorphic up to a null set} if there exist invertible measure-preserving maps $\varphi,\psi:[0,1]\to[0,1]$ such that $U$ and
$$
W^{\varphi,\psi}(x,y)=W(\varphi(x),\psi(y))
$$
are equal almost everywhere.

\end{enumerate}

\end{definition}

\subsection{Graphs as graphons}

\begin{definition}\label{def-stepfun}
A \emph{stepfunction} is a kernel $W$ such that there exist partitions $[0,1]=\coprod_{i=1}^n U_i$ and $[0,1]=\coprod_{j=1}^m V_j$ into measurable sets so that $W$ is constant on each $U_i\times V_j$. 
\end{definition}

We often informally refer to the number $\max(m,n)$ as the number of \emph{steps} of the stepfunction as above. 

\begin{definition}
A \emph{weighted bipartite graph} $\Gamma=(U,V,E,w)$ consists of  a bipartite graph $(U,V,E\subseteq U\times V)$ such that
\begin{enumerate}
\item for $i\in U$, we have a positive real weight $w_i$;
\item for $j\in V$, we have a positive real weight $w_j$;
\item for each edge $(i,j)\in E$, we have a weight $w_{ij}\in \R$. 
\end{enumerate}
The above assignment can be extended to $U\times V$ by stipulating that $w_{ij}=0$ whenever $(i,j)$ is not an edge. 
\end{definition}

\begin{definition}\label{assoc-stepfun}
Let $\Gamma=(U,V,E\subseteq U\times V,w)$ be a finite weighted bipartite graph. 
We define the associated stepfunction
$$
W(\Gamma)
$$ 
by considering the partitions $[0,1]=\coprod_{i\in U} U_i$ and $[0,1]=\coprod_{j\in V} V_j$ with lengths
$\mu(U_i)=w_i/w_U$ and $\mu(V_j)=w_j/w_V$, where $w_U=\sum_{i\in U}w_i$, $w_V=\sum_{j\in V}w_j$,  and by letting 
$$
W(\Gamma){\restriction}_{U_i\times V_j}=w_{ij}
$$
for $i\in U$, $j\in V$.
\end{definition}
By a slight abuse of notation, whenever we wish to consider a graph $\Gamma$ in the space of graphons, we implicitly identify $\Gamma$ with $W(\Gamma)$.

\begin{notation}\label{assoc-incid}
Let $\Gamma$ be a finite \emph{edge-weighted} bipartite graph, i.e., a weighted bipartite graph with all node weights 1, and let $A=(w_{ij})$ be its matrix of edge weights. We write
$$
W(A)=W(\Gamma). 
$$
In particular, we use the notation $$W(w)$$ for the constant kernel with value $w$.
\end{notation}



\subsection{Kernel operators}

\begin{definition}
The \emph{kernel operator} $T_W:L^1[0,1]\to L^\infty[0,1]$ associated to a kernel $W\in\W$ is defined by
$$
(T_W f)(x)=\int_0^1W(x,y)f(y) dy.
$$ 
\end{definition}

\begin{remark}
Considered as an operator $L^2[0,1]\to L^2[0,1]$, $T=T_W$ is a Hilbert-Schmidt operator; it is a compact operator, with a \emph{singular value decomposition}
$$
T(f)=\sum_i\sigma_i\langle f,u_i\rangle v_i,
$$
where $\{u_i\}$ and $\{v_i\}$ are orthonormal sets and $\sigma_i$ are positive with $\sigma_i\to 0$ such that the \emph{Hilbert-Schmidt norm} satisfies
$$
\norm{T}_2^2=\tr(T^* T)=\sum_i\sigma_i^2=\norm{W}_{L^2([0,1]^2)}^2<\infty.
$$

The spectrum of $T$ is discrete, and the nonzero eigenvalues $\lambda_i$ satisfy $\lim_i\lambda_i=0$.
If $W$ is symmetric, the eigenvalues $\lambda_i$ are real and we have a spectral decomposition 
$$
W(x,y)\sim \sum_k \lambda_k f_k(x)f_k(y),
$$
where $f_k$ is the normalised eigenfunction corresponding to the eigenvalue $\lambda_k$.
\end{remark}

\subsection{Operations on kernels}

\begin{definition}
Let $W_{i}$ be a countable family of kernels, and let $a_i$ and $b_i$ be positive real numbers with $\sum_i a_i=1$ and $\sum_ib_i=1$. The \emph{direct sum of kernels $W_{i}$ with weights $(a_i,b_i)$}, denoted
$$
W=\oplus_{i}(a_i,b_i)W_{i},
$$
is defined as follows. We partition the interval $[0,1]$ into intervals $I_i$ of lengths $a_i$ and also into intervals $J_i$ of lengths $b_i$. We consider the monotone affine maps $\varphi_i$ mapping $I_i$ onto $[0,1]$, and $\psi_i$ mapping $J_i$ onto $[0,1]$, and we let
$$
W(x,y)=\begin{cases} W_i(\varphi_i(x),\psi_i(y)), & \text{if } x\in I_i, y\in J_i\\
0 & \text{otherwise.}
\end{cases}
$$
\end{definition}
A kernel is said to be \emph{connected}, if it is not isomorphic up to a null set to a non-trivial direct sum of kernels.

Apart from the obvious linear structure, we consider the following operations on $\W$.
\begin{definition}
Let $U$ and $W$ be two kernels.
\begin{enumerate}
\item Their \emph{product} is the kernel
$$
(UW)(x,y)=U(x,y)W(x,y).
$$
\item Their \emph{operator product} is the kernel
$$
(U\circ W)(x,y)=\int U(x,z)W(z,y)\, dz.
$$
\end{enumerate}
\end{definition}
\begin{remark}
If $U$ and $W$ are kernels, then, considering the associated kernel operators as operators on $L^2[0,1]$, we have
$$
T_W^*=T_{W^*}, \ \ \ \text{ and }\ \ \  T_{U\circ W}=T_U\, T_W.
$$
\end{remark}

\subsection{The cut norm and distance}

\begin{definition}\label{cut-norm-metric}
The \emph{cut norm} on the linear space $\W$ of kernels is defined by
$$
\norm{W}_\square=\sup_{S,T\subseteq[0,1]}\left|\int_{S\times T}W(x,y)\,dx\,dy\right|,
$$
where $S$ and $T$ vary over all measurable subsets of $[0,1]$. 

The associated \emph{cut metric} is 
$$
d_\square(U,W)=\norm{U-W}_\square.
$$
\end{definition}

\begin{fact}\label{fact:estimate}
For $W\in\W_1$, we have the inequalities
$$
\norm{W}_\square\leq\norm{W}_1\leq\norm{W}_2\leq\norm{W}_\infty\leq 1.
$$
\end{fact}

\begin{definition}
Let $S_{[0,1]}$ denote the group of all invertible measure-preserving maps $[0,1]\to[0,1]$. The \emph{cut distance} between kernels $U$ and $W$ is
$$
\delta_\square(U,W)=\inf_{\varphi,\psi\in S_{[0,1]}} d_\square(U,W^{\varphi,\psi}).
$$
\end{definition}

\subsection{Regularity/homogeneity and the cut distance}

\begin{definition}\label{hom-reg}
Let $\Gamma=(U,V,E\subseteq U\times V)$ be a finite bipartite graph and $\epsilon>0$.
\begin{enumerate}
\item We say that $\Gamma$ is \emph{$\epsilon$-homogeneous of density $w\in[0,1]$} provided, for every $A\subseteq U$ and $B\subseteq V$,
$$
\left| \, \, | E\cap (A\times B)|-w|A||B| \, \,\right|\leq \epsilon |U||V|.
$$
\item We say that $\Gamma$ is \emph{$\epsilon$-regular of density $w\in[0,1]$} if, for every $A\subseteq U$ with $|A|>\epsilon|U|$ and $B\subseteq V$ with $|B|>\epsilon|V|$,
$$
\left | \, \, |E\cap (A\times B)|-w|A||B| \, \,\right|\leq \epsilon |A||B|.
$$
\end{enumerate}
\end{definition}

The well-known connection between regularity lemmas (the existence of a homogeneous partition) and the cut metric (see the discussion following 9.3 in \cite{lovasz}) is explained in the following lemma.

\begin{lemma}\label{regularity}
Let $\Gamma=(U,V,E\subseteq U\times V)$ be a finite bipartite graph, $\epsilon>0$.
\begin{enumerate}
\item 
The graph $\Gamma$ is $\epsilon$-homogeneous with density $w\in[0,1]$ if and only if $d_\square(W(\Gamma),W(w))\leq\epsilon$.
\item Suppose that there exist partitions
$U=\coprod_{i=1}^n U_i$ and $V=\coprod_{j=1}^m V_j$ so that 
$\Gamma\restriction (U_i\times V_j)$ is $\epsilon$-homogeneous with density $w_{ij}$. Let $\bar{\Gamma}$ be a weighted bipartite graph on $U$, $V$ such that all vertices have weight 1 and all edges between $U_i$ and $V_j$ have weight $w_{ij}$. Then
$$
d_\square(W(\Gamma),W(\bar{\Gamma}))\leq\epsilon. 
$$
\end{enumerate}
\end{lemma}
\begin{proof}
Item (1) follows from definitions of homogeneity, associated graphon and cut distance.  Suppose that $\Gamma$ satisfies the assumptions of (2). If we write $\bar{U}_i$ and $\bar{V}_j$ for segments in $[0,1]$ corresponding to $U_i$ and $V_j$ in $W(\Gamma)$, using (1), we obtain that, for measurable $A_i\subseteq \bar{U}_i$ and $B_j\subseteq\bar{V}_j$, 
$$
\left|\int_{A_i\times B_j}W(\Gamma)(x,y)-w_{ij}\, dx\, dy\right|\leq \epsilon \mu(\bar{U}_i)\mu(\bar{V}_j).
$$
Hence, for measurable $A,B\subseteq[0,1]$, we obtain
\begin{multline*}
\left|\int_{A\times B}W(\Gamma)(x,y)-W(\bar{\Gamma})\, dx\, dy\right| \\=
\left|\sum_{i,j}\int_{A_i\times B_j}W(\Gamma)(x,y)-w_{ij})\, dx\, dy\right| 
 \leq \sum_{i,j} \epsilon \mu(\bar{U}_i)\mu(\bar{V}_j)=\epsilon.
\end{multline*}
\end{proof}

\subsection{The space of graphons}

The cut distance is only a pseudo-metric on $\W_0$, so here we work with $\widetilde{\W}_0$, the metric space of classes of graphons at non-zero distance.

\begin{theorem}
The space $(\widetilde{\W}_0,\delta_\square)$ is compact.
\end{theorem} 
\begin{proof}
For symmetric graphons, this is a known theorem of Lov\'asz and Szegedy \cite[Theorem 5.1]{lov-sze}, which essentially follows from a variant of Szemer\'edi's regularity lemma. Our version can be deduced from it through the 
closed embedding from \ref{comp-symm},
%
%
which maps $\W_0$ onto a closed subset of the space of symmetric graphons.
\end{proof}

\section{Background: Asymptotic classes of finite structures}


\subsection{The yoga of definable sets}
Definability in classes of finite structures, in the absence of a monster model, is a somewhat delicate matter. While all this is standard for a logician, we hope to improve the exposition for a reader with a combinatorics background by adopting the following notation inspired by category theory.

Let $\cC$ be an arbitrary category of structures for a fixed first-order language $\cL$ with (substructure) embeddings as morphisms, and let $\cC^\prec$ be the subcategory in which the morphisms are elementary embeddings.

Given  a first-order formula $\varphi(x)$ in the language $\cL$ in variables $x=x_1,\ldots,x_n$, we define the assignment
$$
\tilde{\varphi}:\cC\to\Set, \ \ \ \ \tilde{\varphi}(F)=\varphi(F)=\{a\in F^n\,:\,F\models\varphi(a)\},
$$
mapping a structure $F$ in $\cC$ to the set of realisations of the formula $\varphi$ in $F$.

If $\varphi$ is a sentence in the language $\cL$, then
$$
\tilde{\varphi}:\cC\to\Set
$$
assigns to each $F\in\cC$ the truth value of $\varphi$ in $F$.

For $n>0$, the `affine space' $\A^n$ is the assignment
$$
F\mapsto F^n,
$$
associated with the trivial formula $\land_{i=1}^n (x_i=x_i)$. 

The space $\A^0$ is the constant assignment
$$
F\mapsto\{\bot,\top\}.
$$

For each formula $\varphi$, the restriction of $\tilde{\varphi}$ to $\cC^\prec$ is a functor 
$$
\tilde{\varphi}:\cC^\prec\to\Set.
$$ 

On the other hand, while $\A^n:\cC\to\Set$ is a functor, the assignment $\tilde{\varphi}:\cC\to\Set$ for an arbitrary formula $\varphi(x)$ in variables $x=x_1,\ldots,x_n$ (or a sentence) is at best a \emph{subassignment} of $\A^n$, i.e., for all $F\in \cC$, 
$$
\tilde{\varphi}(F)\subseteq \A^n(F)=F^n.
$$

\begin{definition}\label{def-set}
A subassignment $\bS$ of some $\A^n$ with $n\geq 0$ is called a \emph{definable set} if it is equivalent to the assignment $\tilde{\varphi}$ associated with some first-order formula $\varphi(x)$ in $n$ variables.
\end{definition}
We often emphasise that such sets are `parameter-free', or `defined with no parameters'.

\begin{remark}
If $\bX$ and $\bY$ are definable sets, their argument-wise cartesian product, denoted
$$
\bX\times \bY
$$
is clearly definable.

If $\bX,\bY\subseteq \A^n$ are definable, so are the following sets
$$
\bX\cap \bY,\ \ \ \bX\cup \bY,\ \ \ \bX\setminus \bY.
$$
\end{remark}

\begin{definition}\label{def-function}
A \emph{definable function} 
$$
\bpi:\bX\to \bY
$$
between definable sets $\bX$ and $\bY$ is given through a definable subset $\bGamma\subseteq \bX\times \bY$ such that, for every $F\in\cC$, $\bGamma(F)\subseteq \bX(F)\times \bY(F)$ defines a function $$\bpi_F:\bX(F)\to \bY(F).$$
\end{definition}

\begin{remark}
If $\bpi:\bX\to \bY$ is a definable function, the \emph{image} $\bpi(\bX)$ defined by
$$
\bpi(\bX)(F)=\bpi_F(\bX(F))
$$
is a definable subset of $\bY$.
\end{remark}


\begin{definition}
A \emph{definable function} 
$$
\bff:\bX\to E
$$
from a definable set $\bX$ to an arbitrary set $E$ is determined by a choice of \begin{enumerate}
\item finitely many values $e_1,\ldots,e_n\in E$, and
\item finitely many definable subsets $\bX_i\subseteq \bX$, $i=1,\ldots,n$,
\end{enumerate}
so that, for every $F\in\cC$,
$$
\bX(F)=\amalg_i \bX_i(F)
$$
and 
$\bff_F:\bX(F)\to E$ is defined by
$$
\bff_F\restriction_{\bX_i(F)}=e_i.
$$
We stipulate that, for $\bX=\A^0$, the $\bX_i$ are associated with sentences $\varphi_i$ such that, for every $F\in\cC$, 
$$
\bff_F=e_i \text{ if and only if } F\models \varphi_i
$$
\end{definition}

In the following, we discuss definable sets with \emph{parameters}. We start by a construction allowing us to choose parameters from a fixed structure in our class, and continue onto constructions of definable parameter spaces.

\begin{definition}\label{def-set-params}
Let $F\in\cC$, and let $\cL_F$ be the language obtained by adding the constant symbols for the elements of $F$ to $\cL$. Let $\cC_F$ be the subcategory of $\cC$ consisting of superstructures of $F$. 

If $\varphi(x;c)$ is a formula in the language $\cL_F$ with parameters $c=c_1,\ldots,c_m\in F$, then we can define the assignment
$$
\tilde{\varphi}_c:\cC_F\to \Set,
$$
which to each $F'\supseteq F$ assigns the set of realisations $\varphi(F',c)$ in $F'$.
\end{definition}

\begin{definition}\label{rel-def-set}
Let $\bS$ be a definable set. A \emph{definable set $\bX$ over $\bS$} (or, \emph{with parameters in $\bS$})
 is a definable function in the sense of \ref{def-function}
$$
\bpi:\bX\to \bS.
$$
A \emph{definable map of definable sets over $\bS$} is a commutative diagram
\begin{center}
 \begin{tikzpicture} 
 [cross line/.style={preaction={draw=white, -,
line width=3pt}}]
\matrix(m)[matrix of math nodes, 
inner sep=2pt, 
row sep=1em, column sep=.7em, text height=1.5ex, text depth=0.25ex,ampersand replacement=\&]
 { 
 |(gl)|{\bX} 	\&  \& |(gd)| {\bY}    \\
\&        |(d)|{\bS}  \&\\};
\path[->,font=\scriptsize,>=to, thin]
(gl) edge   (gd) edge  (d)
(gd) edge  (d) 
;
\end{tikzpicture}
\end{center}
\end{definition}

\begin{definition}
Given definable sets $\bpi:\bX\to \bS$ and $\bpi':\bX'\to \bS$ over $\bS$,  their \emph{fibre product} $\bX\times_{\bS}\bX'$ is
\begin{multline*}
\bX\times_{\bS}\bX'(F)=\bX(F)\times_{\bS(F)}\bX'(F)=\\ 
\{(x,x')\in \bX(F)\times \bX'(S):\bpi(x)=\bpi'(x')\},
\end{multline*}
and it is again a definable set over $\bS$.

Given $s\in \bS(F)$ for some $F\in\cC$, We can consider the singleton $\{s\}$ as a definable set with parameters from $F$ over $\bS$, so we obtain a definable set 
$$
\bX_s=\bX\times_{\bS}\{s\}:\cC_F\to\Set,
$$
which is called the \emph{fibre of $\bX$ over $\bS$ with parameter $s$}, and it is a definable set over $F$ in the sense of \ref{def-set-params} via
$$
\bX_s(F')=\bpi_{F'}^{-1}(s),
$$
for every $F'$ in $\cC_F$. 
\end{definition}

\begin{remark}
By the above, a definable set $\bX$ over $\bS$ gives rise to a family of definable sets $\bX_s$ parametrised by parameters $s$ from $\bS$. 

In model theory, this object is usually called a \emph{uniformly definable family of definable sets}. The reason for our specific formulation is the additional precision needed to treat parameters over a family of structures, as opposed to working in a fixed `monster model', which is the most familiar setting for model theory. 

Note, if $\bX$ is a definable set, it can naturally be considered as a definable set over $\A^0$. Indeed, we consider the definable map which takes $\bX(F)$ to $\top$ if and only if $\bX(F)\neq\emptyset$.

\end{remark}

\begin{definition}
Let $\bX$ be a definable set over $\bS$ and let $E$ be a set. A \emph{definable function
$$
\bff:\bX\to E
$$
over $\bS$} on a class of structures $\cC$ is determined by a choice of
\begin{enumerate}
\item finitely many definable functions $\be_1,\ldots,\be_n:\bS\to E$, and
\item finitely many definable sets $\bX_1,\ldots,\bX_n$ over $\bS$
\end{enumerate}
such that, for every $F\in\cC$, every $s\in\bS(F)$, we have
$$
\bX_{1,s}\amalg\cdots\amalg\bX_{n,s}=\bX_s,
$$
and 
$$
\bff_s:\bX_s\to E 
$$
is given on $\cC_F$ by
$$
\bff_{s,F'}\restriction_{\bX_{i,s}(F')}=\be_i(s).
$$
\end{definition}

\subsection{Counting and asymptotic classes}

\begin{definition}\label{CDM}
Let $\cC$ be a class of finite structures (considered a category with substructure embeddings).
We say that $\cC$ is a \emph{CDM-class}, if, for every definable set $\bX$ over $\bS$, there exist
\begin{enumerate}
\item a definable function $\bmu_{\bX}:\bS\to \Q$,
\item a definable function $\dd_{\bX}:\bS\to \N$,
\item a constant $C_{\bX}>0$,
\end{enumerate}
so that, for every $F\in\cC$ and every $s\in \bS(F)$,  
$$
\left| |\bX_s(F)|-\bmu_{\bX}(s)|F|^{\dd_{\bX}(s)}\right|\leq C_{\bX} |F|^{\dd_{\bX}(s)-1/2}. 
$$
\end{definition}

\begin{definition}\label{asymp-cl-def}
Let $\cC$ be a class of finite structures (considered a category with substructure embeddings).
We say that $\cC$ is an \emph{asymptotic class} (in the sense of \cite{dugald-charlie} and \cite{dugald-richard}), if, for every definable set $\bX$ over $\bS$, there exist
\begin{enumerate}
\item a definable function $\bmu_{\bX}:\bS\to \Q$,
\item a definable function $\dd_{\bX}:\bS\to \N$,
\end{enumerate}
so that, for every $\epsilon>0$ there exists a constant $N>0$ such that for every $F\in\cC$ with $|F|>N$ and every $s\in \bS(F)$,  
$$
\left| |\bX_s(F)|-\bmu_{\bX}(s)|F|^{\dd_{\bX}(s)}\right|\leq \epsilon |F|^{\dd_{\bX}(s)}. 
$$
\end{definition}

\begin{definition}\label{relative-cdm}
Let $\cC$ be a class of structures with a given function
$$
\chi:\cC\to \N.
$$
We say that $\cC$ is a \emph{CDM-class relative to $\chi$}, if for every definable 
set $\bX$ over $\bS$, there exist
\begin{enumerate}
\item a definable function $\bmu_{\bX}:\bS\to \Q\cup\{\infty\}$,
\item a definable function $\dd_{\bX}:\bS\to \N\cup\{\infty\}$,
\item a constant $C_{\bX}>0$,
\end{enumerate}
so that, for every $F\in\cC$ and every $s\in \bS(F)$,  
$$
\left| |\bX_s(F)|-\bmu_{\bX}(s)\chi(F)^{\dd_{\bX}(s)}\right|\leq C_{\bX} \chi(F)^{\dd_{\bX}(s)-1/2},
$$
where we stipulate that $\bmu_{\bX}(s)<\infty$ is and only if $\dd_{\bX}(s)<\infty$.

We define an \emph{asymptotic class relative to $\chi$} analogously, to reflect the error term from \ref{asymp-cl-def}.
\end{definition}

\begin{terminology}\label{fin-reldim}
With notation of \ref{relative-cdm}, we may informally refer to the number $\bmu_{\bX}(s)$ as \emph{measure}, and to the number $\dd_{\bX}(s)$ as \emph{dimension}.

We say that $\bX\to \bS$ is of \emph{finite relative dimension}, provided $\dd_{\bX}$ maps into $\N$.
\end{terminology}

\begin{remark}\label{actual-use-cdm}
Suppose $\bX$ is of finite relative dimension over $\bS$, over a class $\cC$ that is an asymptotic class relative to $\chi$. For each definable subset $\bY$ of $\bX$ over $\bS$, we obtain a definable \emph{probability function}
$$
\bmu_{\bY/\bX}:\bS\to \Q, \ \ \ \ \bmu_{\bY/\bX}(s)=
\begin{cases} 
\frac{\bmu_{\bY}(s)}{\bmu_{\bX}(s)}& \mbox{if }\dd_{\bY}(s)=\dd_{\bX}(s)\\
0 &\mbox{otherwise}.
\end{cases}
$$
By the above definitions, for $F$ ranging over $\cC$ and $s\in \bS(F)$, the expressions 
$$
\frac{|\bY_s(F)|}{|\bX_s(F)|}-\bmu_{\bY/\bX}(s)
$$
are 
$o(1)$ in $\chi(F)$, i.e., converge to 0 as $\chi(F)\to\infty$. If $\cC$ is a CDM-class relative to $\chi$, then the above expressions are 
$O(\chi(F)^{-1/2})$.

 If $\bY$ is lower-dimensional than $\bX$, then the above expressions are $O(\chi(F)^{-1})$ .
\end{remark}

\begin{remark}
A (relative) asymptotic class is clearly a (relative) CDM-class. 
A CDM-class is a relative CDM-class of finite structures with respect to the cardinality function such that all definable sets are finite dimensional. Similarly, an asymptotic class is a relative asymptotic class of finite structures with respect to the cardinality function such that all definable sets are finite dimensional. 
\end{remark}

\begin{example}
The class of finite fields was shown to be a CDM-class in the foundational paper \cite{CDM} by Chatzidakis-van den Dries-Macintyre. 

Following a discussion of difference fields in Section~\ref{sect:frob}, we will explain in \ref{ac-with-frob} how fields with powers of Frobenius constitute a relative CDM-class by the main theorem of \cite{ive-mark}.

For further examples of CDM-classes and a discussion of differences between CDM and asymptotic classes we refer the reader to \cite{dugald-charlie}.
\end{example}

\section{A weak regularity lemma}

In this section we show that an iterate of a graphon can be regularised in infinity norm. It is a general result for graphons and, although it does not use definability, it will serve as an important first step in the proof of our main theorem later. 

\begin{lemma}\label{wk6-reg}
Let $W$ be a graphon. For every $\epsilon\in (0,1)$ there exists a stepfunction $W'$ with $N(\epsilon)\leq (3/\epsilon^3)^{(1/\epsilon^2)}$ steps such that, writing $W^6=W\circ W^*\circ W\circ W^*\circ W\circ W^*$,
$$
\norm{W^6-W'}_\infty\leq 2\epsilon^2.
$$
\end{lemma}

\begin{proof} We reformulate the first part of the proof of \cite[Lemma~3]{tao-blog} in the language of graphons.

Let $T=T_W:L^2[0,1]\to L^2[0,1]$ be the kernel operator 
$$
T(f)(v)=\int_{[0,1]}W(u,v)f(u)\,du
$$
associated with the graphon $W$. Its adjoint $T^*$ is given by
$$
T^*(g)(u)=\int_{[0,1]}W(u,v)g(v)\,dv.
$$

Cauchy-Schwarz inequality yields that for all $f\in L^2([0,1])$,
\begin{equation}\label{tao5} 
\norm{Tf}_2\leq\norm{Tf}_\infty\leq\norm{f}_2, 
\end{equation}
and similarly, for $g\in L^2[0,1]$,
\begin{equation}\label{tao6} 
\norm{T^*g}_2\leq\norm{T^*g}_\infty\leq\norm{g}_2.
\end{equation}
We can apply the singular value decomposition to the Hilbert-Schmidt operator $T$, which gives
$$
Tf=\sum_i\sigma_i\langle f,u_i\rangle_2 y_i
$$
and
$$
T^*g=\sum_i\sigma_i\langle g,y_i\rangle_2u_i
$$
for some sequence $\sigma_i$ of singular values with $\sigma_1\geq\sigma_2\geq\cdots>0$, where $u_i$ and $y_i$ are orthonormal systems in $L^2[0,1]$. 

The operator $TT^*:L^2[0,1]\to L^2[0,1]$ can be diagonalised as
$$
TT^* g=\sum_i\sigma_i^2\langle g,y_i\rangle_2 y_i,
$$
whence
$$
\tr(TT^*)=\sum_i\sigma_i^2.
$$
On the other hand, we obtain the Hilbert-Schmidt norm bound
\begin{equation}\label{tao7}
\sum_i\sigma_i^2=\tr(TT^*)=\int |W(x,y)|^2\, dxdy\leq 1.
\end{equation}
Using 
$$
y_i=\frac{1}{\sigma_i}Tu_i, \ \ \ u_i=\frac{1}{\sigma_i}T^*y_i,
$$
as well as (\ref{tao5}), (\ref{tao6}), we obtain
\begin{equation}\label{tao8}
\norm{y_i}_\infty\leq\frac{1}{\sigma_i}, \ \ \text{ and } \ \ \ \norm{u_i}_\infty\leq\frac{1}{\sigma_i}.
\end{equation}

We use these bounds to find a low rank approximation to the sixth power 
$$
TT^*TT^*TT^*:L^2[0,1]\to L^2[0,1].
$$
Intuitively, taking a high power `tames' any unpredictable behaviour of $T$ and produces a more manageable operator. 

The above operator can be diagonalised as 
\begin{equation}\label{tao9}
TT^*TT^*TT^* g=\sum_i\sigma_i^6\langle g,y_i\rangle_2 y_i.
\end{equation}

Given an $\epsilon>0$, we split
$$
TT^*TT^*TT^*=A_{\epsilon}+B_{\epsilon},
$$
where $A=A_{\epsilon}$ is a low rank operator
$$
Ag=\sum_{i:\sigma_i\geq\epsilon}\sigma_i^6 \langle g,y_i\rangle_2 y_i,
$$
and $B=B_{\epsilon}$ is the error term
$$
Bg=\sum_{i:\sigma_i<\epsilon}\sigma_i^6 \langle g,y_i\rangle_2 y_i,
$$
Using the triangle inequality, H\"older's inequality \ref{holder-ineq} and (\ref{tao8}), (\ref{tao7}), for any $g\in L^1[0,1]$, 
\begin{align*}
\norm{Bg}_\infty & \leq \sum_{i:\sigma_i<\epsilon}\sigma_i^6 |\langle g,y_i\rangle_2|\frac{1}{\sigma_i}\leq 
\sum_{i:\sigma_i<\epsilon}\sigma_i^6 \norm{g}_1\norm{y_i}_\infty\frac{1}{\sigma_i}
\\ 
&\leq \sum_{i:\sigma_i<\epsilon}\sigma_i^6 \norm{g}_1\frac{1}{\sigma_i^2} \leq \epsilon^2\sum_i\sigma_i^2\norm{g}_1\leq \epsilon^2\norm{g}_1.
\end{align*}
Let $$
\delta=\frac{\epsilon^2}{3}.
$$
Using (\ref{tao8}), 
we discretise 
$$
y_i=y'_{i,\epsilon}+e_{i,\epsilon},
$$
where $y_i'=y'_{i,\epsilon}$ takes at most $1/\sigma_i\delta$ values, $\norm{y_i'}_\infty\leq\norm{y_i}_\infty$, and $e_i$ is bounded in magnitude by $\delta$.
We split $A_{\epsilon}=A'_{\epsilon}+E_{\epsilon}$, where 
$$
A'g=\sum_{i:\sigma_i\geq\epsilon}\sigma_i^6 \langle g,y'_i\rangle_2 y'_i
$$
and
$$
Eg=\sum_{i:\sigma_i\geq\epsilon}\sigma_i^6 \left(\langle g,y'_i\rangle_2 e_i+\langle g,e_i\rangle_2 y_i'+\langle g,e_i\rangle_2 e_i\right).
$$
By the choice of $\delta$, 
and arguments analogous to the above, we get that
\begin{align*}
\norm{Eg}_\infty & \leq \sum_{i:\sigma_i\geq\epsilon}\sigma_i^6\left(\norm{g}_1\norm{y_i'}_\infty\norm{e_i}_\infty+\norm{g}_1\norm{e_i}_\infty\norm{y_i'}_\infty+\norm{g}_1\norm{e_i}_\infty^2\right)\\
& \leq \norm{g}_1 \sum_{i:\sigma_i\geq\epsilon} \sigma_i^6\,\left(2\frac{\delta}{\sigma_i}+\delta^2\right)
\leq \norm{g}_1 \sum_{i:\sigma_i\geq\epsilon} 2\sigma_i^5\delta+\sigma_i^6\delta^2
\leq \epsilon^2\norm{g}_1.
\end{align*}
%
Thus, we have decomposed
$$
TT^*TT^*TT^*=A'_{\epsilon}+E'_{\epsilon},
$$
where $A'$ is of `low rank', and $E'=E+B$ has integral kernel bounded pointwise by $2\epsilon^2$.

Using (\ref{tao7}), the number of summands in the definition of $A$ is at most $1/\epsilon^2$. We partition 
$$
[0,1]=V_1\cup\ldots\cup V_n,
$$ 
where $V_j=V_{j,\epsilon}$ are the intersections of the level sets of the $y_i'$ (removing any empty cells to ensure the $V_j$ are all non-empty), and 
$$
n=N(\epsilon)\leq (1/\epsilon\delta)^{(1/{\epsilon^2})}=(3/\epsilon^3)^{(1/\epsilon^2)}.
$$
The (sought-after) integral kernel $W'$ of $A'$ is constant on each $V_j\times V_k$, so the integral kernel of
$TT^*TT^*TT^*$ fluctuates by at most $2\epsilon^2$ on each $V_j\times V_k$.
\end{proof}

\section{Algebraic regularity lemma}


\begin{definition}
A \emph{definable bipartite graph over $\bS$} is a triple $\bGamma=(\bU,\bV,\bE)$, where 
 $\bU$, $\bV$, $\bE\subseteq \bU\times \bV$ are definable sets over $\bS$. For each $F\in\cC$ and each point $s\in \bS(F)$, 
we obtain a bipartite graph
$$
\bGamma_s=\bGamma_s(F)=(\bU_s(F),\bV_s(F),\bE_s(F)).
$$
\end{definition}

We are interested in describing the (limit) behaviour of the graphs $\bGamma_s(F)$ as $F$ and $s$ vary over an asymptotic class $\cC$.

\begin{definition}
Let $\bU$, $\bV$ be definable sets over $\bS$. A  \emph{definable stepfunction} on $\bU\times\bV$ over $\bS$ is a definable function $\bW:\bU\times \bV\to\R$ over $\bS$ such that there exist
definable sets $\bU_1, \ldots, \bU_m$ and  $\bV_1,\ldots, \bV_n$ over $\bS$ so that, for each $F\in\cC$ and each $s\in \bS(F)$,
$$
\bU_s=\amalg_{i=1}^n \bU_{i,s}\ \ \ \ \text{ and }\ \ \ \ \ \bV_s=\amalg_{j=1}^m \bV_{j,s},
$$
and
$\bW_s$ is constant on each $\bU_{i,s}\times\bV_{j,s}$.

For $F\in\cC$ and $s\in\bS(F)$, by a slight abuse of notation, we will often identify the weighted graph  $\bW_s(F)$ with its associated stepfunction $W(\bW_s(F))\in\W_0$.
\end{definition}

\begin{lemma}\label{acc-pts-def-stepf}
Let $\bW$ be a definable stepfunction on $\bU\times\bV$ over $\bS$ of finite relative dimension on an asymptotic class $\cC$ relative to $\chi$.
The set of accumulation points of the set
$$
\{\bW_s(F): F\in\cC,\, s\in\bS(F)\}
$$
in the space $\tilde{\W_0}$ of graphons  is a finite set of graphons represented by stepfunctions.
\end{lemma}
\begin{proof}
By definition, there exist definable sets $\bU_1, \ldots, \bU_m$ and  $\bV_1,\ldots, \bV_n$ over $\bS$ partitioning $\bU$ and $\bV$ and definable functions $\be_{ij}:\bS\to [0,1]$ so that
$$
\bW_s\restriction_{\bU_{i,s}\times\bV_{j,s}}=\be_{ij}(s).
$$ 
Using \ref{actual-use-cdm} and the fact that we are dealing with the finite relative dimension, there exist definable functions 
$$
\bmu_{\bU_i/\bU}:\bS\to \Q\ \ \ \text{ and }\ \ \ \ \bmu_{\bV_j/\bV}:\bS\to \Q$$ 
$$
\frac{|\bU_{i,s}(F)|}{|\bU_{s}(F)|}-\mu_{\bU_i/\bU}(s)=o(1),
$$
and analogously for $\bV_j$.

We can partition 
$$
\bS=\bS_1\amalg\cdots\amalg \bS_r
$$
so that, for all $s\in \bS_l(F)$, 
$$
\be_{ij}(s)=e_{ijl}\in [0,1],\ \ \
\bmu_{\bU_i/\bU}(s)=\mu_{il}\in\Q,\ \ \ 
\bmu_{\bV_j/\bV}(s)=\nu_{jl}\in\Q.
$$
Thus, using \ref{actual-use-cdm}, up to isomorphism, the stepfunction associated to $\bW_s(F)$ through \ref{assoc-stepfun} is within 
$o(1)$ from a stepfunction associated to a weighted bipartite graph with vertex weights $\{\mu_{il}:i\}$ and $\{\nu_{jl}:j\}$ and edge weights $\{e_{ijl}:i,j\}$ in $\norm{\cdot}_1$-norm, and hence, by Fact \ref{fact:estimate}, in the cut norm. 

In fact, the net of graphons associated to the graphs $\{\bW_s(F):s\in \bS_l(F)\}$ with respect to the preorder  induced by $\chi(F)$ has a stepfunction as a limit in the  $\norm{\cdot}_1$-norm.
\end{proof}

%
%

\begin{theorem}[Tao's algebraic regularity lemma for asymptotic classes]\label{reg-asymp}
Let $\bGamma=(\bU,\bV,\bE)$ be a definable bipartite graph of finite relative dimension over a definable set $\bS$ on an asymptotic class $\cC$ relative to $\chi$. 
Then there exists a definable set $\tilde{\bS}$ over $\bS$ and a definable stepfunction $\bW$ over $\tilde{\bS}$ such that for every $\varepsilon>0$, there exists an $M>0$ such that for every $F\in\cC$ with $\chi(F)\geq M$, every $\tilde{s}\in \tilde{\bS}(F)$ mapping onto $s\in \bS(F)$, 
$$
d_\square(\bGamma_s(F),\bW_{\tilde{s}}(F))\leq\varepsilon.
$$ 
\end{theorem}

\begin{theorem}[Tao's algebraic regularity lemma]\label{reg-cdm}
With assumptions from \ref{reg-asymp}, suppose that $\cC$ is a CDM-class relative to $\chi$. 
There exists a constant $M=M(\bGamma)>0$, a definable set $\tilde{\bS}$ over $\bS$ and a 
definable stepfunction $\bW$ over $\tilde{\bS}$ such that for every $F\in\cC$, every $\tilde{s}\in \tilde{\bS}(F)$ mapping onto $s\in \bS(F)$, 
$$
d_\square(\bGamma_s(F),\bW_{\tilde{s}}(F))\leq M\chi(F)^{-1/12}.
$$ 
\end{theorem}

We will prove Theorem~\ref{reg-cdm} because of the more interesting/precise analysis of the error term, and the proof of Theorem~\ref{reg-asymp} follows along the same lines. We follow the ideas from the proof of \cite[Lemma~3]{tao-blog}. 
\begin{proof} For simplicity of notation, we will write the proof for a CDM-class. Note that, because of the assumption of finite relative dimension, the same proof will work for a relative CDM-class by replacing every instance of $|F|$ by $\chi(F)$, for $F\in\cC$.

The weak regularity result established in Proposition~\ref{wk6-reg} states that for every graphon $W$ and every $\epsilon>0$, there exists a  stepfunction with at most $N(\epsilon)=(3/\epsilon^3)^{(1/\epsilon^2)}$ steps which approximates $W^6$ up to $2\epsilon^2$ in the $\norm{\cdot}_\infty$-norm.

The idea is to improve this result by using the definability of $\bGamma$ and the constraints on the growth rates of the sets of realisations of definable sets over CDM-classes in the spirit of \ref{actual-use-cdm}.

Let us name the key objects. 
For $F\in\cC$ and an $s\in \bS(F)$, let $T=T_{s}:L^2[0,1]\to L^2[0,1]$ be the kernel operator 
$$
T_{s}(f)(v)=\int_{[0,1]}W_{s}(u,v)f(u)\,du
$$
associated with the stepfunction $W_{s}=W(\bGamma_s(F))$. 

The proof consists of the following conceptual steps.

\begin{enumerate}
\item By the definability of $\bGamma$, the set of relevant CDM-growth rates of certain definable invariants of $\bGamma$ is finite and hence separated by some minimal distance $\delta$. We choose an $\epsilon>0$ small enough with respect to $\delta$, and we construct a \emph{definable stepfunction} $\bW$ with at most $N(\epsilon)$ steps anticipating all the possible behaviours that may occur in regularising each $\bGamma_s(F)$ for varying $F\in\cC$ and $s\in\bS(F)$.

\item We verify that, for any large enough $F\in\cC$, any $s\in \bS(F)$, writing $W_s=W(\bGamma_s(F))$,  any 
(not necessarily definable) stepfunction with at most $N(\epsilon)$ steps that regularises $W_s^6$ and $(W_s^*)^6$ in view of \ref{wk6-reg}, turns out to be close to the steps of a realisation of $\bW$ constructed in (1) in the supremum norm.

\item From (1) and (2), it follows that $W_s$ is close to a realisation of $\bW$ in the cut metric, which shows the required regularity.
\end{enumerate}

\noindent{\bf Step 1.} We will use the definability of $\bGamma$ and decide on an appropriate choice of $\epsilon$.

The integral kernel $K_3(v,v')=K_{3,s}(v,v')$ of $TT^*TT^*TT^*$ for $T=T_{s}$ is explicitly given as 
$$
K_3(v,v')(F)=\frac{|\bG_{v,v',s}(F)|}{|\bU_s(F)|^3|\bV_s(F)|^2},
$$ 
where $\bG_{v,v',s}\subseteq \bU_s\times \bV_s\times \bU_s\times \bV_s\times \bU_s$ is the definable set
\begin{multline*}
\bG_{v,v',s}=\{(u_1,v_2,u_2,v_3,u_3)\in \bU_s\times \bV_s\times \bU_s\times \bV_s\times \bU_s:\\
(u_1,v), (u_1,v_2), (u_2,v_2), (u_2,v_3), (u_3, v_3), (u_3,v')\in \bE_s\}.
\end{multline*}
Using \ref{CDM}, we get that there exists a 
a definable function $\bc=\bc(v,v',s)$ symmetric in $v,v'$ 
such that, for all $F\in\cC$, $s\in \bS(F)$ and $v,v'\in \bV_s(F)$,  
\begin{equation}\label{def-k3}
K_{3,s}(v,v')=\bc(v,v',s)+O(|F|^{-1/2}).
\end{equation}

Replacing the role of $T_s$ by $T^*_s$, we consider the integral kernel
$$
K_{3,s}^*(u,u')
$$
of $T^*TT^*TT^*T$, and we similarly find a definable function $\bc^*=\bc^*(u,u',s)$ symmetric in $u,u'$ 
such that for all $F$, $s\in \bS(F)$ and $u,u'\in \bU_s(F)$,
\begin{equation}\label{unnamed}
K^*_{3,s}(u,u')=\bc^*(u,u',s)+O(|F|^{-1/2}).
\end{equation}

The set of values of $\bc$ and $\bc^*$ is finite, so its elements are separated by some minimal distance $\delta>0$. 

Let us choose an $\epsilon\in (0,1)$ such that $$2\epsilon^2<\delta/2.$$ 

For every $\bar{n}\leq n\leq N(\epsilon)=(3/\epsilon^3)^{(1/\epsilon^2)}$, 
writing
$$
\bV_{i,(v_1,\ldots,v_n,s)}=\{v\in \bV_s:\bigwedge_j \bc(v,v_j,s)=\bc(v_i,v_j,s)\},
$$
the formula
\begin{align*}
\psi_{\bar{n},n}(v_1,\ldots,v_n,s)\equiv & 
\left(\bigwedge_i v_i\in \bV_s \right)   
\land \left(\bV_s=\amalg_{i\leq n} \bV_{i,(v_1,\ldots,v_n,s)}\right)\\
&\land \bigwedge_{i\leq\bar{n}}\dim(\bV_{i,(v_1,\ldots,v_n,s)})=\dim(\bV_s) \\
&\land \dim(\bV_s\setminus\amalg_{i\leq\bar{n}}\bV_{i,(v_1,\ldots,v_n,s)})<\dim(\bV_s)\\
&\land \bigwedge_i\bigwedge_j \forall v\in \bV_{i,\gamma}\ \forall v'\in \bV_{j,\gamma}\ \ \bc(v,v',s)=\bc(v_i,v_j,s)
\end{align*}
expresses that $s'=(v_1,\ldots,v_n,s)$ can serve as parameters for a definable stepfunction on $\bV_s$ with steps $\bV_{1,s'},\ldots,\bV_{n,s'}$ and values $\bc(v_i,v_j,s)$ on $\bV_{i,s'}\times \bV_{j,s'}$,  where $\bV_{1,s'},\ldots,\bV_{\bar{n},s'}$ are the top-dimensional steps that we will call \emph{large} in the sequel.

Let $\bS_{\bar{n},n}\to \bS$ be the definable set associated with the formula $\psi_{\bar{n},n}$, and let 
$$
\bS_V\to \bS
$$
be the disjoint union of all the $\bS_{\bar{n},n}$ for $\bar{n}\leq n\leq N(\epsilon)$.
For each $i\leq n$, we have a definable set
$$
\bV_{i,\bar{n},n}\to \bS_{\bar{n},n}\to \bS_V.
$$

Analogously, using $\bc^*$ in place of $\bc$, for every $\bar{m}\leq m\leq N(\epsilon)$ 
we construct definable sets
$$
\bU_{j,\bar{m},m}\to \bS_{\bar{m},m}^*\to \bS_U,
$$
for $j\leq m$, which form steps of a definable stepfunction with values $\bc^*$. 

Using \ref{CDM} again, for each $i,\bar{n},n$ and $j,\bar{m},m$ as above, there is a definable function $\bff_{j,\bar{m},m,i,\bar{n},n}:\bS^*_{\bar{m},m}\times_{\bS} \bS_{\bar{n},n}\to\Q$ such that for all $F\in\cC$, and $s'\in \bS^*_{\bar{m},m}(F)$, $s''\in \bS_{\bar{n},n}(F)$ mapping onto $s\in \bS(F)$,
\begin{equation}\label{end-st1}
\frac{\left|\bE_s(F)\cap \left(\bU_{j,\bar{m},m,s'}(F)\times \bV_{i,\bar{n},n,s''}(F)\right)\right|}{\left|\bU_{j,\bar{m},m,s'}(F)\times \bV_{i,\bar{n},n,s''}(F)\right|}=\bff_{j,\bar{m},m,i,\bar{n},n}(s',s'')+O(|F|^{-1/2}).
\end{equation}

Let 
$$
\tilde{\bS}=\bS_U\times_\bS\bS_V.
$$


We consider the definable stepfunction 
$
\bW
$
over $\tilde{\bS}$ which is given by
%
$$\bW\restriction_{\bU_{j,\bar{m},m}\times \bV_{i,\bar{n},n}}=\bff_{j,\bar{m},m,i,\bar{n},n}$$
over the component $\bS^*_{\bar{m},m}\times_{\bS}\bS_{\bar{n},n}$ of $\tilde{\bS}$.


\noindent{\bf Step 2.} Let us choose $M$ such that 
\begin{quote}
for $|F|\geq M$,  the error terms in \ref{def-k3}, \ref{unnamed}, \ref{end-st1} are less than $\delta/2$.
\end{quote}

Let us consider $F\in\cC$ with $|F|\geq M$, a parameter $s\in\bS(F)$ and the stepfunction $W_s=W(\bGamma_s(F))$.

We apply Proposition~\ref{wk6-reg} to $W_s$ for the $\epsilon$ chosen in Step~1, and this yields a partition 
$$
\bV_s(F)=V_1\amalg \cdots\amalg V_n
$$
into $n\leq N(\epsilon)$ sets $V_i$ so that the integral kernel $K_3(v,v')$ fluctuates by less than $\delta/2$ on each $V_i\times V_j$. On the other hand, by (\ref{def-k3}), and the choice of $F$, each value of $K_3(v,v')$ is within $\delta/2$ from a value of the definable function $\bc$. 
Hence we can improve the near-constancy to the property
\begin{equation}\label{tao10}
K_{3,s}(v,v')=\gamma_{ij}+O(|F|^{-1/2})
\end{equation}
for some value $\gamma_{ij}$ of $\bc$, whenever $v\in V_i$, $v'\in V_j$ and $i,j\in\{1,\ldots,n\}$.

Note that $K_3$ is symmetric, so $\gamma_{ij}$ will be symmetric too.  Moreover, by reducing the number of $V_i$, we can assume that $\gamma_{ij}=\gamma_{i'j}$ for all $j\in\{1,\ldots,n\}$ implies $i=i'$, i.e., that the array $\gamma$ is reduced.

Using this property, we see that, if $v_1,\ldots,v_n$ satisfy the definable condition
$$
\bigwedge_i\bigwedge_j \bc(v_i,v_j,s)=\gamma_{ij},
$$
then there exists a permutation $\sigma$ of $\{1,\ldots,n\}$ such that $\gamma_{\sigma(i),\sigma(j)}=\gamma_{ij}$ and $v_i\in V_{\sigma(i)}$.

Note, if $v_i\in V_i$ for $i\in\{1,\ldots,n\}$ are parameters, then $V_i$ is definable as
$$
V_i=\{v\in \bV_s(F):\bigwedge_j \bc(v,v_j)=\gamma_{ij}\}.
$$
Hence, choosing parameters $s'=(v_1,\ldots,v_n,s)\in\bS_{\bar{n},n,s}(F)$ 
determines $V_i$ up to a permutation $\sigma$ preserving $\gamma_{i,j}$. Moreover, for a fixed $s'$ we may relabel the $V_i$ so that 
$$
V_{i}=\bV_{i,\bar{n},n,s'}(F).
$$

Note that $\bV_{i,\bar{n},n}$ for $i\leq\bar{n}$ are large, so \ref{CDM}
gives that for some positive $a_i\in\Q$,
\begin{equation}\label{largecells}
\frac{|V_i|}{|\bV_s(F)|}=a_i+O(|F|^{-1/2}),
\end{equation}
while for $i>\bar{n}$,
 $\dim(\bV_{i,\bar{n},n,s'})<\dim(V)$ and
\begin{equation}\label{smallcells}
|V_i|\ll |F|^{-1}|\bV_s(F)|.
\end{equation}

Let $\bar{V}_j\subseteq[0,1]$ denote the segment corresponding to a large step $V_j$ in the stepfunction $W_s$ associated to the graph $\bGamma_s(F)$ via \ref{assoc-stepfun}. Suppose $g\in L^2[0,1]$ is supported on $\bar{V}_j$ with mean zero. Combining (\ref{tao9}) and (\ref{tao10}), we obtain
\begin{align*}
\sum_i \sigma_i^6|\langle g,y_i\rangle|^2 & =
|\langle TT^*TT^*TT^*g,g\rangle|=\left|\int_0^1\int_0^1 K_3(v,v')g(v)g(v')\, dv\, dv'\right|\\ 
& \ll |F|^{-1/2}\norm{g}^2_2.
\end{align*}
On the other hand, Bessel's inequality gives
$$
\sum_i |\langle g,y_i\rangle|^2\leq \norm{g}^2_2,
$$
and so H\"older's inequality (combining these two inequalities via \ref{holder-lemma}) shows that
$$
\norm{T^*g}^2_2=\sum_i\sigma_i^2|\langle g,y_i\rangle|^2\ll |F|^{-1/6}\norm{g}^2_2,
$$ 
whence $$
\norm{T^*g}_2\ll |F|^{-1/12}\norm{g}_2,
$$

Similarly, arguing for $K_3^*(u,u')$, we can find a partition 
$$\bU_s(F)=U_1\amalg\cdots\amalg U_{m}$$ with $m\leq N(\epsilon)$ 
so that, upon fixing a parameter $s''=(u_1,\ldots,u_m,s)\in\bS^*_{\bar{m},m,s}(F)$, 
by relabeling $U_j$, we have
$$
U_j=\bU_{j,\bar{m},m,s''}(F).
$$
Moreover,
$$
\norm{Tf}_2\ll |F|^{-1/12}\norm{f}_2,
$$
when $f$ is supported on one of the segments $\bar{U}_i$ corresponding to a large step $U_i$ with mean zero.

\noindent{\bf Step 3.}
Let $\bW$ be the definable stepfunction over $\tilde{S}$ defined in Step~1.


We claim that for $\tilde{s}\in\tilde{S}(F)$ which maps to $s'$ and $s''$ from Step~2,
$$
d_{\square}(W_{s},\bW_{\tilde{s}}(F))\ll |F|^{-1/12}.
$$

Combining the estimates obtained in Step 2 with the Cauchy-Schwarz inequality, we deduce that
$$
\left|\langle Tf,g\rangle_2\right|\ll |F|^{-1/12}\norm{f}_2\norm{g}_2
$$ 
whenever $f$ and $g$ are supported on $\bar{V}_i$ and $\bar{U}_j$ with at least one of $f$, $g$ of mean zero.

For measurable $A\subseteq \bar{U}_i$ and $B\subseteq\bar{V}_j$, we decompose the characteristic functions 
$$1_A=(1_A-\mu(A)/\mu(\bar{U}_i)1_{\bar{U}_i})+\mu(A)/\mu(\bar{U}_i)1_{\bar{U}_i}$$ and
$$1_B=(1_B-\mu(B)/\mu(\bar{V}_j)1_{\bar{V}_j})+\mu(B)/\mu(\bar{V}_j)1_{\bar{V}_j}$$ into sums of mean-zero and constant functions.
Note that the mean-zero parts are bounded by 1. We obtain that
$$
\langle T1_A,1_B\rangle_2=\frac{\mu(A)}{\mu(\bar{U}_i)}\frac{\mu(B)}{\mu(\bar{V}_i)}
\langle T1_{\bar{U}_i},1_{\bar{V}_j}\rangle_2+\sqrt{\mu(\bar{U}_i)\mu(\bar{V}_j)}\,O(|F|^{-1/12}).
$$

On the other hand, using \ref{end-st1},
\begin{align*}
\langle T1_{\bar{U}_i},1_{\bar{V}_j}\rangle_2 & =\int_{[0,1]^2}W(u,v)1_{\bar{U}_i}(u)1_{\bar{V}_j}(v)\,du dv\\
& = \frac{|\bE_s\cap (U_i\times V_j)|}{|U||V|} =(\bff_{ij}(\tilde{s})+O(|F|^{-1/2}))\mu(\bar{U}_i)\mu(\bar{V}_j),
\end{align*}
hence
$$
\langle T1_{A},1_{B}\rangle_2=\mu(A)\mu(B)\bff_{ij}(\tilde{s})+\sqrt{\mu(\bar{U}_i)\mu(\bar{V}_j)}\,O(|F|^{-1/12}).
$$
Thus, for measurable $A\subseteq \bar{U}_i$ and $B\subseteq\bar{V}_j$, we have
$$
\int_{[0,1]^2}(W_s(u,v)-\bff_{ij}(\tilde{s}))1_A(u)1_B(v)\ll \sqrt{\mu(\bar{U}_i)\mu(\bar{V}_j)}|F|^{-1/12}.
$$
For arbitrary measurable $A,B\subseteq[0,1]$, writing $A_i=A\cap \bar{U}_i$ and $B_j=B\cap\bar{V}_j$, 
\begin{align*}
d_\square(W_s,\bW_{\tilde{s}}(F)) & = 
\sup_{A,B\subseteq[0,1]}\int_{[0,1]^2}(W_s-\bW_{\tilde{s}}(F))1_A1_B\\
&=\sup_{A,B\subseteq[0,1]}\sum_{i,j}\int_{[0,1]^2}(W_s-\bff_{ij}(\tilde{s}))1_{A_i}1_{B_j}+  O(|F|^{-1})  \\
&\ll \sum_{i,j}|F|^{-1/12}\sqrt{\mu(\bar{U}_i)\mu(\bar{V}_j)}  +  O(|F|^{-1}) \\
& \leq  \sqrt{\bar{m}\bar{n}}|F|^{-1/12}+  O(|F|^{-1}) ,
\end{align*}
where the summation is over the large blocks and, by (\ref{smallcells}), and the behaviour over non-large/lower-dimensional blocks is absorbed by the $O(|F|^{-1})$ term. The last inequality follows from Jensen's inequality.

The desired conclusion follows from the fact that $\bar{m}$ and $\bar{n}$ are bounded independently of $F$.


%


\end{proof}

Combining \ref{acc-pts-def-stepf} and \ref{reg-asymp}, we obtain the following.

\begin{corollary}\label{accumul}
Let $\bGamma$ be a definable graph of finite relative dimension over $\bS$ on an asymptotic class $\cC$ relative to $\chi$. The set of accumulation points of the family of finite graphs 
$$
\{\bGamma_s(F): F\in\cC, s\in \bS(F) \}
$$
in the space $\tilde{\W_0}$ of graphons has a set of representatives which is a finite set of stepfunctions.


\end{corollary}

\section{Finite fields}
The original algebraic regularity lemma by Tao in \cite{tao-original} was formulated for graphs definable in the language of rings over finite fields. We give several examples of graphons arising in this context by along the lines of our graphon formulation \ref{accumul}.

\begin{example}\label{paley}
Consider the definable graph $\bGamma=(\bU,\bV,\bE)$ where $\bU=\bV=\A^1$, and the edge relation is
$$
\bE(x,y)\equiv \exists z\, \, x+y=z^2.
$$
We are interested in accumulation points in the graphon space of the set of graphs
$$
\bGamma(\F_q),
$$
as $\F_q$ ranges over all finite fields (these are in fact a symmetric variant of the well-known \emph{Paley graphs}). When $\chr(\F_q)\neq 2$, being a square is an event of `CDM probability' (in the sense of \ref{actual-use-cdm}) 1/2. Thus, the edge density is approximately 1/2 and these graphs accumulate around the constant graphon
$W(1/2)$. 
These graphs are interesting in model theory due to the fact that their ultraproduct limit is the \emph{random graph}.  

For $\chr(\F_q)=2$, everything is a square, so the graphs accumulate around the constant graphon $W(1)$. 

The definable stepfunction $\bW:\bU\times\bV\to\Q$ that regularises $\bGamma$ is 
$$
\bW(x,y)=
\begin{cases} 
1/2 & \text{if }\exists z\, z+z\neq 0;\\
1 & \text{if } \forall z\, z+z=0.
\end{cases}
$$

Hence, the set of accumulation points is
$$
\{W(1/2),W(1)\}.
$$
\end{example}

\begin{example}
Consider the definable graph $\bGamma=(\A^1,\A^1,\bE)$ with the edge relation
$$
\bE(x,y)\equiv \exists z\, \, xy=z^2.
$$
The definable sets
$$
\bU_0(x)\equiv \exists z\, x=z^2 \ \ \ \text{ and }\ \ \ \ \bU_1=\A^1\setminus \bU_0
$$
partition $\A^1$. Let $\bW$ be the definable function which returns $1$ on $\bU_0\times\bU_0\cup\bU_1\times\bU_1$ and $0$ elsewhere. The definable stepfunction $\bW$ regularises $\bGamma$, and the set of accumulation points of $\bGamma(\F_q)$ in the space of graphons
is
$$
\left\{W\begin{pmatrix} 1 & 0\\ 0 & 1\end{pmatrix}, W(1)\right\},
$$
where the former denotes the stepfunction associated to the graph with edge weights matrix 
$\begin{pmatrix} 1 & 0\\ 0 & 1\end{pmatrix}$ as in \ref{assoc-incid}, which is the accumulation point of $\bGamma(\F_q)$ for odd $q$, and the latter is the constant graphon $1$, which is the accumulation point of $\bGamma(\F_q)$ for even $q$.
\end{example}

In the following example, unlike in Example~\ref{paley}, the limit behaviour is not determined by the characteristic only.
\begin{example}
Let us define the graph $\bGamma$ on $\A^1$ through the edge relation
$$
\bE(x,y)\equiv \exists z\, \, x+y=z^3.
$$
 In fact,  the set of cubes in $\F_q$ is of size $(q-1)/3+1$ if $3| q-1$ and of size $q$ otherwise. 

The first condition happens precisely when the primitive cube root of unity belongs to $\F_q$,  which is expressible by the first-order condition
$$
\exists t\, \,t^3=1\land t\neq 1
$$
so $\bGamma$ is definably regularised.
To conclude, the graphs $\bGamma(\F_q)$ accumulate around
$
W(1/3)
$
when $3| q-1$, and around 
$
W(1)
$
otherwise.
\end{example}

The following example illustrates the need to introduce parameters.
\begin{example}\label{cubes-prod}
Consider the graph $\bGamma$ on $\A^1$ with the edge relation
$$
\bE(x,y)\equiv \exists z\,\, xy=z^3.
$$
The graph is trivially considered over the base $\bS=\A^0$. Let 
$$\tilde{S}(z)\equiv z^3=1\land z\neq 1,$$
and consider the definable sets $\bU_i$ over $\tilde{\bS}$, $i=0,1,2$:
$$
\bU_i(x;z)\equiv \exists u \,(x=z^iu^3).
$$
We see that $\bGamma$ is regularised by the definable stepfunction $\bW(x,y;z)$ over $\tilde{\bS}$ taking value $1$ on $\bU_0\times\bU_0\cup\bU_1\times\bU_2\cup\bU_2\times\bU_1$ and $0$ elsewhere, and the graphs $\bGamma(\F_q)$ accumulate around the stepfunction
$$
W\begin{pmatrix} 1&0&0\\ 0 & 0 &1\\ 0 & 1& 0\\\end{pmatrix}
$$
when $3|q-1$, and around
$
W(1)
$
when $3\not| q-1$.

\end{example}

\section{Fields with powers of Frobenius}\label{sect:frob}

\subsection{Difference algebra}
A \emph{difference ring} is a pair
$$
(R,\sigma)
$$
consisting of a commutative ring with identity $R$ and an endomorphism $\sigma:R\to R$. We call it a \emph{difference field} if $R$ happens to be a field.

A \emph{homomorphism} of difference rings 
$$
f:(R,\sigma_R)\to (S,\sigma_S)
$$
is a ring homomorphism $f:R\to S$ satisfying
$$
\sigma_S\circ f=f\circ\sigma_R.
$$
The category of difference rings is denoted by
$$
\sigma\Rng
$$
A difference ring $(R,\sigma)$ is a \emph{transformal domain} if $R$ is a domain and $\sigma$ is injective. 

\begin{notation}
For a prime power $q$, we write
$$K_q=(\bar{\F}_q,\varphi_q)$$
for the difference field consisting of the algebraic closure of a finite field, together with a power of the Frobenius automorphism
$$
\varphi_q(x)=x^q.
$$
\end{notation}

\subsection{Counting points over fields with Frobenius}

In \cite{ive-mark}, we established the following difference analogue of \cite{CDM}.

\begin{theorem}
Let $\bX\to \bS$ be a definable map in the language of difference rings.
Then there exist
\begin{enumerate}
\item a definable function $\bmu_{\bX}:\bS\to \Q\cup\{\infty\}$,
\item a definable function $\dd_{\bX}:\bS\to \N\cup\{\infty\}$,
\item a constant $C_{\bX}>0$,
\end{enumerate}
so that, for every $K_q$ with $q>C_{\bX}$ and every $s\in\bS(K_q)$,
$$
\left| |\bX_s(K_q)|-\bmu_{\bX}(s)\,q^{\dd_{\bX}(s)}\right|\leq C_{\bX}\, q^{\dd_{\bX}(s)-1/2}. 
$$
\end{theorem}


\begin{corollary}\label{ac-with-frob}
The class $\cC$ of fields $K_q$ for prime powers $q$ is a CDM-class relative to the function
$$
\chi(K_q)=q.
$$
\end{corollary}

\begin{remark}
An alternative approach to \ref{ac-with-frob}  is to study \emph{finite fields} with powers of Frobenius, which requires the study of $N$-dimensional asymptotic classes. The authors of \cite{GMS} go even further, and note that Tao's regularity lemma applies in classes of structures with \emph{pseudofinite dimension}.
\end{remark}

\subsection{Algebraic regularity lemma for fields with Frobenius}

Using the fact \ref{ac-with-frob} that fields with Frobenius constitute a relative CDM-class and regularity lemma \ref{reg-cdm}, we obtain the following.
\begin{corollary}
\label{reg-frob}
Let $\bGamma=(\bU,\bV,\bE)$ be a definable bipartite graph of finite relative dimension over a definable set $\bS$ in the language of difference rings. 
There exists a constant $M=M(\bGamma)>0$, a definable set $\tilde{\bS}$ over $\bS$ and a 
definable stepfunction $\bW$ over $\tilde{\bS}$ such that for every $K_q$ with $q>M$, every $\tilde{s}\in \tilde{\bS}(K_q)$ mapping onto $s\in \bS(K_q)$, 
$$
d_\square(\bGamma_s(K_q),\bW_{\tilde{s}}(K_q))\leq M\,q^{-1/12}.
$$ 
\end{corollary}

\begin{example}\label{diff-cubes}
Let $\bU(x)\equiv x\sigma(x)^2=1$, and let $\bGamma$ be a graph on $\bU\times\bU$ with the edge relation
$$
\bE(x,y)\equiv \exists z\in\bU\, \, xy=z^3.
$$
Then $\bU(K_q)$ is the group $\mu_{2q+1}$ of $(2q+1)$-th roots of unity in $\bar{\F}_q$, and the size of the set of cubes depends on whether $\mu_3\subseteq\mu_{2q+1}$, i.e., whether $3|2q+1$.
Henceforth we can follow the reasoning from Example~\ref{cubes-prod} and we obtain the same limit stepfunctions. 

On the other hand, let us point out that the graphs $\bGamma(K_q)$ cannot be obtained from a graph interpretable over finite fields, which shows that the difference context is genuinely richer. Indeed, already in even characteristic, for each $q$, the smallest finite field containing the set $\mu_{2q+1}=\bU(K_q)$ is $\F_{4q^2}$, so the size of the set of realisations grows roughly as the square root of the size of the corresponding finite field, which cannot happen by the CDM-property \ref{CDM} proved for finite fields in \cite{CDM}.
\end{example}

\begin{example}
Let $\bGamma$ be defined on the same $\bU$ as in Example~\ref{diff-cubes}, but with 
$$
\bE(x,y)\equiv \exists z\in\bU\, \, x\sigma(y)=z^3.
$$
In order to facilitate the notation, write $\Gamma=(U,U,E)=\bGamma(K_q)$, so that $U=\mu_{2q+1}\subseteq\bar{\F}_q$. Let $\zeta\in \bar{\F}_q$ be a primitive cube root of unity, and write $U_i=\zeta^iU^3$, for $i=0,1,2$. We split the fields $K_q$ into subclasses:
\begin{enumerate}
\item $\zeta\notin U$, i.e. $3\not| 2q+1$;
\item $\zeta\in U$ and $\zeta^q=\zeta$, i.e., $3|2q+1$ and $3|q-1$;
\item $\zeta\in U$ and $\zeta^q=\zeta^2$, i.e., $3|2q+1$ and $3\not|q-1$.
\end{enumerate}
The set of corresponding limits is
$$
\left\{W(1),
W\begin{pmatrix} 1&0&0\\ 0 & 0 &1\\ 0 & 1& 0\end{pmatrix}, 
W\begin{pmatrix} 1 & 0 & 0\\0 & 1 &0\\0&0&1\end{pmatrix}\right\}.
$$

\end{example}


\section{Expander difference polynomials}

\subsection{Difference schemes}

\begin{definition}
If $(k,\sigma)$ is a difference ring, a \emph{difference $(k,\sigma)$-algebra} is a difference ring $(A,\sigma)$ endowed with a difference ring homomorphism $(k,\sigma)\to (A,\sigma)$. A morphism between $(k,\sigma)$-algebras $(A,\sigma)$ and $(B,\sigma)$ is a difference ring homomorphism $(A,\sigma)\to (B,\sigma)$ which commutes with the structure maps $(k,\sigma)\to (A,\sigma)$ and $(k,\sigma)\to (B,\sigma)$. We write
$$
(k,\sigma)\Alg
$$ 
for the category of $(k,\sigma)$-algebras.
\end{definition}

\begin{definition}
Let $(k,\sigma)$ be a difference ring. The \emph{difference polynomial ring} in variables $x_1,\ldots,x_n$ is the difference ring in infinitely many variables
$$
k[x_1,\ldots,x_n]_\sigma=k[x_{1,i},\ldots,x_{n,i}: i\geq 0],
$$
with $\sigma$ inferred from the rule $\sigma(x_{j,i})=x_{j,i+1}$ and its action on $k$.
Informally speaking, if $x=(x_1,\ldots,x_n)$, we can write 
$$k[x]_\sigma=k[x,\sigma x,\sigma^2 x,\ldots].$$
Note that $k[x]_\sigma$ is naturally a $(k,\sigma)$-algebra.
\end{definition}

\begin{definition}
Let $(k,\sigma)$ be a difference ring. We say that a $(k,\sigma)$-algebra $A$ is
\begin{enumerate}
\item of \emph{finite $\sigma$-type}, if it admits an epimorphism $k[x_1,\ldots,x_n]_\sigma\to A$, for some $n$;
\item of \emph{finite $\sigma$-presentation}, if it is a quotient $$A\simeq k[x_1,\ldots,x_n]_\sigma/\langle f_1,\ldots,f_m\rangle_\sigma,$$
where $\langle f_1,\ldots,f_m\rangle_\sigma$ is the \emph{difference ideal} (ideal closed under $\sigma$) generated by $f_1,\ldots,f_n\in k[x_1,\ldots,x_n]_\sigma$, for some $m,n$. 

\end{enumerate}
\end{definition}

\begin{definition}
Let $(A,\sigma)$ be a difference ring. The \emph{affine difference scheme} associated with $A$ is the functor
$$
\bX:\sigma\Rng\to\Set,\ \ \  \bX(R,\sigma)=\Hom_{\sigma\Rng}((A,\sigma),(R,\sigma)).
$$ 
A \emph{morphism} between affine difference schemes  is a natural  transformation
$$\varphi:\bX\to \bY.$$
We shall only ever discuss affine difference schemes, so we omit the word {`}affine' form our descriptions. The category of difference schemes is denoted
$$
\sigma\Sch.
$$
\end{definition}
\begin{definition}
Given $\bS\in\sigma\Sch$, a \emph{difference scheme over $\bS$} is a morphism $\bX\to\bS$. A \emph{morphism} between schemes $\bX\to \bS$ and $\bY\to\bS$ is given by a morphism $\bX\to \bY$ which commutes with the structure morphisms to $\bS$, as in the commutative diagram from \ref{rel-def-set}. We obtain the category of difference schemes over $\bS$,
$$
\sigma\Sch_{/\bS}.
$$
\end{definition}
\begin{remark}
By Yoneda's lemma, if $\bX$ is associated with a difference ring $(A,\sigma)$ and $\bY$ with $(B,\sigma)$, a morphism $\bX\to \bY$ always arises from a difference ring homomorphism $(B,\sigma)\to (A,\sigma)$, whence the category of difference schemes is opposite to that of difference rings,
$$
\sigma\Sch\simeq (\sigma\Rng)^\text{op}.
$$
Moreover, if $\bS$ corresponds to a difference ring $(k,\sigma)$, then
$$
\sigma\Sch_{/\bS}\simeq ((k,\sigma)\Alg)^\text{op}.
$$
\end{remark}


\begin{definition}

Let $f:\bX\to \bS$ be a morphism of difference schemes associated with a difference ring homomorphism $(k,\sigma)\to (A,\sigma)$. We say that
\begin{enumerate}
\item $\bX$ is \emph{transformally integral}, if $(A,\sigma)$ is a transformal domain;
\item $f$ is \emph{of finite $\sigma$-type}  if $(A,\sigma)$ is a $(k,\sigma)$-algebra of finite $\sigma$-type;
\item $f$ is \emph{of finite $\sigma$-presentation} over $S$ if $(A,\sigma)$ is a $(k,\sigma)$-algebra of finite $\sigma$-presentation. 
\item $f$ is a \emph{$\sigma$-localisation}, if $A=k[1/a]_{\sigma}$ for some $a\in k$.
\end{enumerate}
\end{definition}

\begin{definition}
Let $f=F(x,\sigma x,\ldots,\sigma^r x)\in k[x]_\sigma$, with $x=(x_1,\ldots,x_n)$, and some ordinary polynomial $F\in k[t_0,t_1,\ldots,t_r]$. The \emph{set of solutions to the equation}
$$
f(x)=0
$$
\emph{in  a difference ring $(R,\varphi)$}  extending $(k,\sigma)$ is defined as
$$
\{a\in R^n:F(a,\varphi a,\ldots,\varphi^r a)=0\}.
$$
\end{definition}

\begin{example}
The set of solutions of the difference equation $\sigma x=x$ in the difference field $K_q=(\bar{\F}_p,\varphi_q)$ is the finite field $\F_q$.
\end{example}

\begin{remark}\label{diffsch-def}
Let $\bS$ be a difference scheme assocatied with a difference ring $(k,\sigma)$, and consider a system of difference polynomial equations 
\begin{align*}
f_1(x_1,\ldots,x_n) & = 0\\
&\vdots  \\
f_m(x_1,\ldots,x_n) &=0,
\end{align*}
where $f_i\in k[x_1,\ldots,x_n]_\sigma$, and consider the difference scheme $\bX$ over $\bS$ associated to the $(k,\sigma)$-algebra of finite $\sigma$-presentation
$$
A=k[x_1,\ldots,x_n]_\sigma/\langle f_1,\ldots,f_m\rangle_\sigma.
$$
For any $(k,\sigma)$-algebra $(R,\sigma)$, the set of \emph{$(R,\sigma)$-rational points}
$$
\bX_{/\bS}(R,\sigma)=\Hom_{(k,\sigma)\Alg}(A,R)
$$
can be identified with the set of solutions of the above system in $(R,\sigma)$, i.e., the intersection of the solution sets of all equations $f_i=0$ in $(R,\sigma)$.

This observation clarifies how a difference scheme of finite presentation can be considered a (quantifier-free) \emph{definable set} as in \ref{def-set}.

\end{remark}

\begin{definition}
Let $f:\bX\to \bS$ be a morphism of difference schemes of finite $\sigma$-presentation over $\Z$, and let $\cC$ be the relative CDM-class of fields with Frobenius $K_q=(\bar{\F}_p,\varphi_q)$.  
We say that 
\begin{enumerate}
\item $f$ is of \emph{finite relative dimension}, if it is such when considered as a definable map over $\cC$ in the sense of \ref{fin-reldim};
\item $f$ is \emph{dominant}, if the dimension of its image as the definable set is of full dimension in $\bS$;
\item $f$ is \emph{finite}, if the map $f(K_q)$ between realisation sets  has finite fibres for all $K_q\in \cC$;
\item $f$ is \emph{generically finite}, if it fails to have finite fibres only over a lower-dimensional subset of $\bS$.
\end{enumerate}  
\end{definition}

\begin{remark}
All of the above properties of difference schemes and their morphisms have intrinsic formulations rooted in difference algebra, as given in \cite{udi-frob} and \cite{ive-mark}. In particular, the notion of \emph{total dimension} of difference schemes is shown to agree with the CDM-dimension in \cite{ive-mark}. We choose these equivalent formulations in order to simplify the presentation.  
\end{remark}

\begin{terminology}
Following the comparison of difference schemes to definable sets made in  \ref{diffsch-def},  we shall no longer use the boldface notation to denote difference schemes.
\end{terminology}

\begin{definition}\label{def-corresp}
Let $X$ and $Y$ be difference schemes. A \emph{correspondence} $X\rightsquigarrow Y$
is given by a difference scheme $W\subseteq X\times Y$ such that the projection $W\to X$ is dominant with generically finite fibres.
\end{definition}

\subsection{Algebraic constraint and group configuration}
The results in this subsection parallel those of Tao in \cite{tao-original}.

\begin{definition}
Let $S$ be a transformally integral difference scheme of finite $\sigma$-presentation over $\Z$ and let $X_1,\ldots,X_n, Y$ be difference schemes of finite $\sigma$-presentation and finite relative 
dimension over $S$, and let 
$$
f:X_1\times_S\cdots \times_S X_n\to Y
$$
be a morphism of difference schemes over $S$.
We say that $f$ is a \emph{moderate asymmetric expander} if there exist constants $c,C>0$ such that, for every difference field $K_q=(\bar{\F}_p,\varphi_q)$, every $s\in S(K_q)$ and every choice of
$A_i\subseteq X_{i,s}(K_q)$ with $|A_i|\geq C |X_{i,s}(K_q)|^{1-c}$, we have
$$
|f_s(A_1,\ldots,A_n)|\geq C^{-1} |Y_s(K_q)|.
$$
\end{definition}

\begin{assumption}[Expansion dichotomy]\label{expansion-dich}
Let $X,Y,Z$ be difference schemes of finite $\sigma$-presentation over a transformally integral $S$ with geometrically transformally integral fibres of finite relative 
dimension, and let 
$$
f:X\times_S Y\to Z
$$
be a morphism of difference schemes over $S$.
Then at least one of the following statements hold:
\begin{enumerate}
\item (Algebraic constraint). The morphism $X\times_S X\times_S Y\times_S Y\to Z\times_S Z\times_S Z\times_S Z$,
$$
(x,x',y,y')\mapsto (f(x,y),f(x,y'),f(x',y),f(x',y'))
$$
is not dominant.
\item (Moderate expansion property).  There exists a $\sigma$-localisation $S'$ of $S$ and a constant $C>0$ such that for every $K_q$, every $s\in S'(K_q)$, and every $A\subseteq X_s(K_q)$, $B\subseteq Y_s(K_q)$, if $|A||B|\geq C q^{-1/8} |X_s(K_q)||Y_s(K_q)|$, then $|f_s(A,B)|\geq C^{-1}|Z_s(K_q)|$.
In particular, the morphism $f$ is a moderate asymmetric expander above $S'$. 
\end{enumerate}
\end{assumption}
We have gone through the steps and methods used in the proof of \cite[Theorem~38]{tao-original} and consider that they can be followed formally to derive the above statement from \ref{reg-frob}. On the other hand, writing out the full proof would require us to repeat lengthy passages of Tao's work, including the \emph{self-improvement trick} from \cite{tao-blog} to obtain the error term of $O(|F|^{-1/4})$ in \ref{reg-frob} and the generalisation of the regularity lemma \ref{reg-frob} to hypergraphs following \cite[Theorem~35]{tao-original}. 

%
%

\begin{theorem}[Solving the algebraic constraint in dimension 1]\label{alg-constraint}
With notation of \ref{expansion-dich}, assume that $X$, $Y$ and $Z$ are of relative 
dimension 1 over $S$ and that the morphism
$X\times_S X\times_S Y\times_S Y\to Z\times_S Z\times_S Z\times_S Z$
given by
$$
(x,x',y,y')\mapsto (f(x,y),f(x,y'),f(x',y),f(x',y'))
$$
is not dominant. Then there exists a generically finite morphism $S'\to S$ and an algebraic group scheme $(A,*)$ over $S'$, which is either $\mathbb{G}_a$, $\mathbb{G}_m$ or of an elliptic curve over $S'$, and correspondences $\chi:X\rightsquigarrow A$, $\upsilon:Y\rightsquigarrow A$ and $\zeta:Z\rightsquigarrow A$ over $S'$ such that
$f$ is in correspondence with the group law on $A$,
$$
\zeta (f(x,y))=\chi(x)*\upsilon(y).
$$
\end{theorem}

The proof requires a vast amount of model theory and it is no longer possible to keep our exposition self-contained, hence we provide a guide to the literature for the interested reader. 
\begin{itemize}
\item Ultraproducts of CDM-classes are \emph{measurable structures} in the sense of \cite{dugald-charlie} and can be treated as \emph{simple theories}, which are governed by a notion of \emph{forking independence}. For background on simple theories and for the basics of \emph{forking calculus} we refer the reader to \cite{casanovas} and \cite{wagner}.
\item For \emph{group configuration} in simple theories, we refer the reader to \cite{gc-btw} and \cite{gc-tw}. 

\item Hrushovski  shows  in \cite{udi-frob} that the limit theory of the relative CDM-class of fields with Frobenius is the theory ACFA of existentially closed difference fields. It is extensively studied in \cite{udi-zoe}.

\item For groups definable in algebraically closed fields, we shall use \cite{bouscaren-weil-chunk}, and for groups definable in ACFA we refer to \cite{anand-piotr}.
\end{itemize}

Martin Bays suggested that our proof can benefit from field-theoretic arguments specific to ACFA. We provide the summary of properties we use, which can all be found in \cite{udi-zoe}.

\begin{fact}\label{acfa-facts}
Let $(\Omega,\sigma)$ be a model of ACFA.
\begin{itemize}
\item[(a)] While ACFA is a simple theory, the underlying field $\Omega$ is algebraically closed, so 
\emph{stability theory} applies to it. 
\item[(b)] The model-theoretic algebraic closure of $A\subset\Omega$, denoted $\acl(A)$, is the field-theoretic algebraic closure of the inversive difference subfield of $\Omega$ generated by $A$. 
\item[(c)] The forking independence is witnessed by algebraic independence in the sense that, for $A,B,C\subset\Omega$, $A\ind_CB$ if and only if the fields $\acl(AC)$ and $\acl(BC)$ are linearly disjoint over $\acl(C)$.
\item[(d)] Let $a\in \Omega$, let $k$ be a difference subfield of $\Omega$, and let $k(a)_\sigma$ be the difference field generated by $a$ over $k$. If $\trdeg(k(a)_\sigma/k)=1$, then $\sigma(a)\in\acl(k(a))$ and thus $a\nind_k\sigma(a)$. This amounts to saying that the difference locus of $a$ over $k$ is of total dimension 1, or, equivalently, of CDM-dimension 1. Such an element also has SU-rank 1 over $k$. Henceforth we shall simply refer to such elements as having \emph{rank 1}. 
\item[(e)] The difference locus of a tuple $a\in \Omega$ of rank 1 elements is dense in its algebraic locus, and hence it is in a correspondence with its algebraic locus in the sense of \ref{def-corresp}.
\item[(f)] Given $a\in\Omega$ and a difference subfield $k$ of $\Omega$, the \emph{canonical basis} $\Cb(a/k)$ is the smallest perfect difference field over which the locus of $a$ over $k$ is defined. It equals, in the pure field language, to the canonical base of the type of $(a,\sigma(a),\ldots)$ over $k$. It has the property that
$a\ind_{\Cb(a/k)}k$. In arbitrary simple theories, canonical bases exist only as `hyperimaginaries', and one is forced to work with \emph{bounded closure} $\bdd$ as opposed to the algebraic closure. In ACFA canonical bases exist in the real word, so any occurrence of $\bdd$ below can be replaced by $\acl$.

\end{itemize}

\end{fact}

\begin{proof}

Our fist objective is to extract a group configuration from the algebraic constraint. We will use the context and notation of partial generic multiactions familiar from \cite{itai}, \cite{ive-th}, \cite{gc-btw} and \cite{gc-tw}, which works in a simple setting, suitable for dealing with arbitrary asymptotic classes whose non-principal ultraproducts are supersimple of finite $SU$-rank. 

An interested reader can reformulate it in purely field-theoretic language using \ref{acfa-facts}.

%
%

The base difference scheme $S$ is associated with a transformal domain $(R,\varsigma)$, so we can take its fraction field $(k,\varsigma)$ and work in a large model $(\Omega,\sigma)$ of ACFA extending it.

Let $\eta$ be a scheme-theoretic generic point of $S$, and  write $X_\eta$, $Y_\eta$ and $Z_\eta$ for the generic fibres of $X$, $Y$ and $Z$ over $S$, considered as difference schemes over $k$. For simplicity of notation, all the independences and bounded/algebraic closures and types we write below will be over $k$.

Let $(x_0,y_0,z_{00})$ in $\Omega$ be a generic point over $k$ on the graph of $f_\eta$, i.e.,  $x_0\in X_\eta$, $y_0\in Y_\eta$ are 
rank 1 points over $k$, with $x_0\ind y_0$ and $z_{00}=f_\eta(x_0,y_0)\in Z_\eta$. 

We may assume that $z_{00}\ind x_0$ and $z_{00}\ind y_0$, which implies that $x_0\in\bdd(y_0,z_{00})$ and $y_0\in\bdd(x_0,z_{00})$. Otherwise, we would have that $z_{00}\in\bdd(x_0)$ or $z_{00}\in\bdd(y_0)$ and we are in a degenerate case where the function $f_\eta$ defines a correspondence dependent on a single variable, which leads to the conclusion of the theorem in a trivial way.


Let  $\pi\subseteq X_\eta\times Y_\eta\times Z_\eta$ be the partial type  over $k$ refining the locus of $(x_0,y_0,z_{00})$ so that
$$
\models\pi(x,y,z)
$$ 
provided $f_\eta(x,y)=z$ and the points $x\in X_\eta$, $y\in Y_\eta$, $z\in Z_\eta$ are pairwise independent.
It is handy to abbreviate this by writing $yz\models x$. Note, if $yz\models x$, then $z\in\dcl(xy)$, $y\in\bdd(xz)$ and $x\in\bdd(yz)$. 
 Thus,  $\pi$ is a \emph{generic invertible multiaction} as in \cite[Definitions 2.1, 2.3]{itai} and \cite[2.1]{ive-th}.

The assumption of non-dominance in this context states that for all $x_0$, $x_1$, $y_0$, $y_1$, $z_{ij}=f_\eta(x_i,y_j)$ (or $y_jz_{ij}\models x_i$),
 the set $\{z_{00},z_{01},z_{10},z_{11}\}$ is dependent.

We claim that 
$$
\overline{\pi\circ\pi^{-1}}
$$
is a generic multiaction in the following sense (cf.~\cite[Definition 2.5, Section~2.3]{itai}, \cite[2.2, 2.3]{ive-th}, \cite[1.2]{gc-tw}). Starting with an independent triple $x_0$, $x_1$, $y_0$, let $y_0z_{00}\models x_0$, $y_0z_{10}\models x_1$, so that $z_{00}z_{10}\models x_1\circ x_0^{-1}$. Let $h=\Cb(z_{00}z_{10}/x_0x_1)\in \overline{\pi\circ\pi^{-1}}$. We need to verify that $h\ind x_0$ and $h\ind x_1$.

Let us choose $y_1\equiv_{x_0x_1}y_0$ such that $y_1\ind_{x_0x_1}y_0$. It follows that the set $\{y_0, y_1, x_0, x_1\}$ is independent. We denote by $z_{01}$ and $z_{11}$ the elements satisfying
$y_1z_{01}z_{11}\equiv_{x_0x_1}y_0z_{00}z_{10}$, which implies that $y_1z_{01}\models x_0$ and $y_1z_{11}\models x_1$.

Using $y_1\ind_{x_0x_1}y_0$, we get that $z_{01}\ind_{x_0x_1}z_{00}z_{10}$, which, together with $z_{01}\ind x_0x_1$ yields $z_{01}\ind x_0x_1z_{00}z_{10}$ and $z_{01}\ind_{z_{00}z_{10}}x_0x_1$. Using the non-dominance assumption, $z_{11}\in\bdd(z_{01}z_{00}z_{10})$, so 
$$
z_{01}z_{11}\ind_{z_{00}z_{10}}x_0x_1.
$$
We conclude that
$$
h=\Cb(z_{00}z_{10}/x_0x_1)=\Cb(z_{01}z_{11}/x_0x_1)=\Cb(z_{01}z_{11}/z_{00}z_{10})\in\bdd(z_{00}z_{10}).
$$
Thus, since $z_{00}z_{10}\ind x_0$ and $z_{00}z_{10}\ind x_1$, we confirm that $h\ind x_0$ and $h\ind x_1$.

In terminology of \cite{itai}, \cite{ive-th}, \cite{gc-btw} and \cite{gc-tw}, it follows that the set of germs of the multiaction $\overline{\pi\circ\pi^{-1}}$ is a polygroup chunk, which entails that the tuple 
$$
(x_0,x_1,h,y_0,z_{00},z_{10})
$$
forms a group configuration, conveniently depicted by the diagram
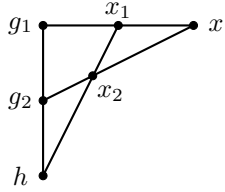
\begin{center}
\begin{tikzpicture}
\draw[fill=black] (0,0) circle (1.5pt);
\draw[fill=black] (0,1) circle (1.5pt);
\draw[fill=black] (0,2) circle (1.5pt);
\draw[fill=black] (1,2) circle (1.5pt);
\draw[fill=black] (2,2) circle (1.5pt);
\draw[fill=black] (.66,1.33) circle (1.5pt);
\node at (-0.3,0) {$h$};
\node at (-0.3,1) {$x_1$};
\node at (-0.3,2) {$x_0$};
\node at (2.3,2) {$y_0$};
\node at (1.0,2.2) {$z_{00}$};
\node at (.9,1.1) {$z_{10}$};
\draw[thick] (1,2)--(0,0) -- (0,2) -- (2,2) -- (0,1);
\end{tikzpicture}
\end{center}
where:
\begin{enumerate}
\item any three non-collinear points are independent;
\item in any triple of collinear points, any two are interbounded over the third;
\item $x_0$ is interbounded with $\Cb(y_0z_{00}/x_0)$, $x_1$ is interbounded with $\Cb(y_0z_{10}/x_1)$ and $h$ is interbounded with $\Cb(z_{00}z_{10}/h)$.
\end{enumerate}


Instead of using the group configuration for simple theories at this point, we prefer to use techniques specific to ACFA to construct a definable group. 

Using the fact that  this is a rank 1 group configuration in $(\Omega,\sigma)$, the same tuple forms a group configuration over $k$ in the underlying algebraically closed field $\Omega$. Indeed, the independence property (1) in the field reduct is clear, and interalgebraicity property (2) follows from \ref{acfa-facts}(b), (d). The property (3) in the reduct is follows from the description of canonical bases \ref{acfa-facts}(f) in the field language specialised to rank 1. 

The (stable) group configuration for algebraically closed fields  yields an ACF-definable group $(G,*)$ over the algebraic closure $\bar{k}$ so that the algebraic locus of $(x_0,y_0,z_{00})$ over $\bar{k}$ is in correspondence with the group law on $G$. Moreover, by results of van den Dries and Hrushovski (described in \cite{bouscaren-weil-chunk}), there is a 1-dimensional algebraic group $A_0$ defined over $\bar{k}$ whose group law is in correspondence with that of $G$.

Using \ref{acfa-facts}, the graph of  $f_\eta$ is in correspondence with the difference locus of $(x_0,y_0,z_{00})$ over $k$, which is in correspondence with the algebraic locus of $(x_0,y_0,z_{00})$ over $k$, which is in correspondence with the locus of the same tuple over $\bar{k}$, which is in correspondence with the group law in $A_0$ over $\bar{k}$. 

Alternatively, we could have argued, using  \cite[Lemma~3.3]{anand-piotr} on the 6-tuple above, that the graph of $f_\eta$ is in correspondence with a difference definable group $H$ over $\bar{k}$, and, by \cite[Theorem~3.1]{anand-piotr}, $H$ is \emph{virtually definably embeddable} into an algebraic group $A$ of dimension 1 over $\bar{k}$.

In either case,  the theorem lists all three possibilities for the 1-dimensional group $A_0$. 
Moreover, $A_0$ is defined over a finite extension $k'$ of $k$, so let $(L,\sigma)$ be the difference extension field of $(k,\varsigma)$ generated by $k'$. Through standard constructibility arguments of $\sigma$-localising $S$ to make it normal, and considering the normalisation of $S$ in $L$, we obtain a generically $\sigma$-finite difference scheme $S'\to S$ and a group scheme $A$ over $S'$ whose generic fibre is $A_0$, and correspondences over $S'$ relating the morphism $f$ and the group law in $A$.

%
%
%
\end{proof}

\begin{remark}
Combining Assumption~\ref{expansion-dich} and Theorem~\ref{alg-constraint}, we obtain the following.

Let $X$, $Y$, $Z$ be difference schemes of finite transformal type and relative total dimension 1 over $S$, and let 
$$
f:X\times_SY\to Z
$$
be a morphism of difference schemes over $S$.
Then at least one of the following statements hold.
\begin{enumerate}
\item The morphism $f$  corresponds to the additive or multiplicative group law in a way described in \ref{alg-constraint}.
\item The morphism $f$ corresponds to the addition law on an elliptic curve in a way described in \ref{alg-constraint}.
\item The morphism $f$ is a moderate asymmetric expander over a $\sigma$-localisation of $S$.
\end{enumerate}
\end{remark}

 

\appendix

\section{Inequalities}

\begin{fact}
Let $H$ be a Hilbert space. We have the following inequalities.
\begin{enumerate}
\item (Cauchy-Schwartz inequality). For $x,y\in H$, we have
$$
|\langle x,y\rangle|\leq \norm{x}\,\norm{y}.
$$
\item (Bessel's inequality). If $e_i$ is an orthonormal sequence in $H$, for every $x\in H$
$$
\sum_i |\langle x,e_i\rangle|^2\leq \norm{x}^2.
$$
\end{enumerate}
\end{fact}

\begin{fact}[H\"older's inequality]\label{holder-ineq}
Let $(X,\mu)$ be a measure space. Let $p, q\in[1,\infty]$ be \emph{H\"older conjugates}, namely satisfying $1/p+1/q=1$. Then, for all measurable real or complex functions $f, g$ on $X$, 
$$
\norm{fg}_1\leq \norm{f}_p\norm{g}_q.
$$  
\end{fact}

Note that, for $p=q=2$, the above yields the Cauchy-Schwartz inequality for $L^2(X,\mu)$. 

\begin{lemma}\label{holder-lemma}
If $a_i$, $b_i$, $u$, $v$ are positive real numbers satisfying
$$
\sum_i a_i^3b_i\leq uv\ \ \ \text{ and }\ \ \ \ \sum_ib_i\leq v,
$$
then
$$
\sum_i a_ib_i\leq u^{\frac{1}{3}}v.
$$
\end{lemma}
\begin{proof}
H\"older's inequality, applied to the space $X=\N$ with the counting measure, states that for real numbers $x_i$, $y_i$, $i\in\N$,  and $p$, $q$ satisfying $1/p+1/q=1$,
$$
\sum_i|x_iy_i|\leq \left(\sum_i|x_i|^p\right)^\frac{1}{p}\left(\sum_i|y_i|^q\right)^\frac{1}{q}.
$$
We can now write
$$
\sum_ia_ib_i=\sum_ia_ib_i^\frac{1}{3}\cdot b^\frac{2}{3}\leq\left(\sum_ia_i^3b_i\right)^\frac{1}{3}\cdot\left(\sum_i(b^\frac{2}{3})^\frac{3}{2}\right)^\frac{2}{3}\leq u^\frac{1}{3}v^\frac{1}{3}\cdot v^\frac{2}{3}=u^\frac{1}{3}v.
$$
\end{proof}

\subsection*{Acknowledgements}
Mirna D\v zamonja received funding from the European's Union Horizon 2020 research and innovation programme under the Maria Sko{\l}odowska-Curie grant agreement No 1010232. She also thanks the Institut d'Histoire et de Philosophie des Sciences et des Techniques, CNRS et Universit\' e Panthéon-Sorbonne, Paris, where she is an Associated Member and the University of East Anglia, Norwich, UK, where she is a Visiting Professor.
The authors thank the referee for a very careful, detailed and helpful referee report.


\end{document}